\documentclass[onefignum,onetabnum]{siamart220329}

\usepackage[utf8]{inputenc}
\usepackage{latexsym,amsfonts,amsmath,amssymb,url}
\usepackage{mathtools}
\usepackage{xcolor,graphicx}
\usepackage{comment}
\usepackage{graphicx,epstopdf} 
\usepackage[caption=false]{subfig} 

\ifpdf
  \DeclareGraphicsExtensions{.eps,.pdf,.png,.jpg}
\else
  \DeclareGraphicsExtensions{.eps}
\fi

\newcommand{\median}{\mathop{\rm median}\limits}

\newcommand{\vol}{\mathrm{vol}}
\newcommand{\rd}{\, \mathrm{d}}
\newcommand{\bsg}{\boldsymbol{g}}
\newcommand{\bsh}{\boldsymbol{h}}
\newcommand{\bsk}{\boldsymbol{k}}

\newcommand{\bszero}{\boldsymbol{0}}
\newcommand{\bsone}{\boldsymbol{1}}
\newcommand{\bsalpha}{\boldsymbol{\alpha}}
\newcommand{\bsgamma}{\boldsymbol{\gamma}}

\newcommand{\bstau}{\boldsymbol{\tau}}
\newcommand{\bsa}{\boldsymbol{a}}
\newcommand{\bsc}{\boldsymbol{c}}
\newcommand{\bsu}{\boldsymbol{u}}
\newcommand{\bsx}{\boldsymbol{x}}
\newcommand{\bsy}{\boldsymbol{y}}

\newcommand{\wor}{\mathrm{wor}}

\newcommand{\Sob}{\mathrm{sob}}

\newcommand{\EE}{\mathbb{E}}
\newcommand{\FF}{\mathbb{F}}
\newcommand{\GG}{\mathbb{G}}
\newcommand{\NN}{\mathbb{N}}
\newcommand{\PP}{\mathbb{P}}
\newcommand{\RR}{\mathbb{R}}

\newcommand{\Scal}{\mathcal{S}}

\DeclareMathOperator{\tr}{tr}
\newcommand{\prob}[1]{\mathrm{Pr}\left[#1\right]}

\newcommand{\cF}{\mathcal{F}}
\newcommand{\cL}{\mathcal{L}}

\newcommand{\cX}{\mathcal{X}}

\newcommand{\Fb}{\mathbb{F}_b}

\newsiamremark{example}{Example}
\newsiamremark{assumption}{Assumption}
\newsiamremark{remark}{Remark}

\allowdisplaybreaks

\headers{A universal median quasi-Monte Carlo integration}{T. Goda, K. Suzuki, and M. Matsumoto}

\title{A universal median quasi-Monte Carlo integration\thanks{Submitted to the editors DATE.
\funding{This work is supported by JSPS KAKENHI Grant Number 20K03744 (T.G.), 23K03210 (T.G.), 20K14326 (K.S.) and 23K03033 (M.M.).}}}

\author{Takashi Goda\thanks{School of Engineering, University of Tokyo, 7-3-1 Hongo, Bunkyo-ku, Tokyo 113-8656, Japan (\email{goda@frcer.t.u-tokyo.ac.jp}).}
\and Kosuke Suzuki\thanks{Graduate School of Advanced Science and Engineering, Hiroshima University, 1-3-1 Kagamiyama, Higashi-Hiroshima, Hiroshima 739-8526, Japan. Current address: Faculty of Pharmacy, Daiichi University of Pharmacy, 22-1 Tamagawa-machi, Minami-ku, Fukuoka 815-8511, Japan (\email{k-suzuki@daiichi-cps.ac.jp}).} \and Makoto Matsumoto\thanks{Graduate School of Advanced Science and Engineering, Hiroshima University, 1-3-1 Kagamiyama, Higashi-Hiroshima, Hiroshima 739-8526, Japan (\email{m-mat@math.sci.hiroshima-u.ac.jp}).}}

\usepackage{amsopn}

\ifpdf
\hypersetup{
  pdftitle={A universal median quasi-Monte Carlo integration},
  pdfauthor={T. Goda, K. Suzuki, and M. Matsumoto}
}
\fi

\begin{document}

\maketitle

\begin{abstract}
We study quasi-Monte Carlo (QMC) integration over the multi-dimensional unit cube in several weighted function spaces with different smoothness classes. We consider approximating the integral of a function by the median of several integral estimates under independent and random choices of the underlying QMC point sets (either linearly scrambled digital nets or infinite-precision polynomial lattice point sets). Even though our approach does not require any information on the smoothness and weights of a target function space as an input, we can prove a probabilistic upper bound on the worst-case error for the respective weighted function space, where the failure probability converges to 0 exponentially fast as the number of estimates increases. Our obtained rates of convergence are nearly optimal for function spaces with finite smoothness, and we can attain a dimension-independent super-polynomial convergence for a class of infinitely differentiable functions. This implies that our median-based QMC rule is \emph{universal} in the sense that it does not need to be adjusted to the smoothness and the weights of the function spaces and yet exhibits the nearly optimal rate of convergence. Numerical experiments support our theoretical results.
\end{abstract}

\begin{keywords}
numerical integration, quasi-Monte Carlo, median, scrambled digital net, polynomial lattice point set, weighted function space, universality
\end{keywords}

\begin{MSCcodes}
65D30, 65D32, 41A55, 42C10
\end{MSCcodes}

\section{Introduction}
We study numerical integration for functions defined over the $s$-dimensional unit cube. For a Lebesgue integrable function $f: [0,1]^s\to \RR$, we denote its integral by
\[ I_s(f):=\int_{[0,1]^s}f(\bsx)\rd \bsx.\]
In this paper, we consider integrands from Banach spaces that have continuous representatives.
Quasi-Monte Carlo (QMC) methods approximate $I_s(f)$ by the equally-weighted mean of function evaluations over a carefully designed point set $P$:
\[ Q_{P}(f):=\frac{1}{N}\sum_{\bsx\in P}f(\bsx),\]
where $P$ is considered to be a multiset in the sense of combinatorics, i.e., a set that allows for repeated elements and takes into account the number of times each element appears, and $N:=|P|$ denotes its cardinality as a multiset. 

The history of QMC methods dates back to the work of Weyl \cite{We16} and has witnessed plenty of works in connection to discrepancy theory. One of the key findings is the Koksma-Hlawka inequality, which states that the integration error $|Q_{P}(f)-I_s(f)|$ is bounded above by the product of the total variation of $f$ in the sense of Hardy and Krause and the star discrepancy of $P$ \cite{Ni92book,DKS13}. This error bound motivates us to study how to design a good point set such that the star discrepancy is small. In fact, there are many explicit constructions of low-discrepancy point sets and sequences provided in the literature, including those by Halton \cite{Ha60}, Sobol' \cite{So67}, Faure \cite{Fa82}, Niederreiter \cite{Ni88}, Tezuka \cite{Te93} and Niederreiter and Xing \cite{NX01bok}. If $f$ has a bounded total variation, the error decays at the rate of either $(\log N)^{s-1}/N$ or $(\log N)^s/N$, depending on whether $P$ is a low-discrepancy point set or the first $N$ points of a low-discrepancy sequence, respectively.

More recently, one of the main approaches in QMC research is to consider a Banach space $B$ with the norm $\|\cdot\|_B$, instead of looking only at functions with bounded variation, and study how to design a point set such that the worst-case error
\[ e^{\wor}(Q_P; B) := \sup_{\substack{f\in B\\ \|f\|_B\leq 1}}\left| Q_{P}(f)-I_{s}(f)\right|\]
becomes small. As we have $|Q_{P}(f)-I_s(f)|\leq \|f\|_B\, e^{\wor}(Q_P; B)$ for any function $f\in B$, a single point set $P$ with a small worst-case error bound should work well for all functions in $B$. Here it is important to note that, as long as we consider the worst-case for inputs $f$ from the unit ball of $B$, which is a convex and symmetric set, optimal algorithms are \emph{non-adaptive} and \emph{linear} in the deterministic worst-case setting, as stated in \cite[Theorem~4.7]{NW08book}. QMC methods are a special case of such non-adaptive, linear algorithms as they simply take the average of function evaluations. If we consider the worst-case over a non-convex or asymmetric input set $F\subset B$, though, then adaption may sometimes help \cite{Nov96}. 

As another key ingredient, Sloan and Wo\'{z}niakowski used \emph{weighted function spaces} to study how the worst-case error depends on the dimension $s$ more precisely \cite{SW98}; we also refer to earlier works by Hickernell \cite{Hic96,Hic98}, who had already introduced weights in the definition of norms for function spaces. There the symbols $\bsgamma=(\gamma_u)_{u\subseteq \{1,\ldots,s\}}$ with $\gamma_\emptyset =1$ control the relative amount in the norm of the function to which each individual variable $x_j$ or each subset of the variables $\bsx_u=(x_j)_{j\in u}$ contributes. Nowadays, these two ingredients have come together and there is a large body of work on \emph{construction of good QMC point sets in weighted function spaces.} Typical examples include rank-1 lattice point sets in weighted Korobov spaces consisting only of smooth periodic functions \cite{Ku03,NC06,DGS22} and (interlaced or extrapolated) polynomial lattice point sets in weighted Sobolev spaces with dominating mixed smoothness $\alpha\geq 2$ containing non-periodic smooth functions \cite{DKLNS14,G15,DGY19}. 

However, there are some shortcomings in this research direction. Firstly, to search for good point sets, the worst-case error or its upper bound must be computable in a reasonable time, which forces us to model the set of $2^s-1$ weights $\gamma_u$ by a smaller number of parameters. Secondly, although the information both on the smoothness and the weights is required as an input to define a weighted function space for construction, it is generally quite a hard problem to find appropriate smoothness and weights for a given application, except for some special cases such as partial differential equations with random coefficients \cite{KSS12,DKLNS14,KN16}. Last, but not least, if good point sets are constructed in a weighted function space to which a target integrand does not belong, we may not have any theoretical guarantee that the error converges at the desired rate. Although these issues have been partly addressed in several works \cite{Di12,DG21,EKNO21}, none of them could completely eliminate the need to specify the weights.

A recent work by L'Ecuyer and the first named author of this paper has introduced a novel method that does not require any knowledge of the smoothness and weights at all but still achieves a desired convergence rate of the worst-case error with a high probability \cite{GL22}. Instead of trying to find good point sets, their approach goes as follows: select rank-1 lattice point sets $P_1, \ldots, P_r$ independently and randomly for an odd integer $r$, compute the integral estimates $Q_{P_1}(f),\ldots,Q_{P_r}(f)$ and take the median of them. Even with such a construction-free rule, a probabilistic worst-case error bound could be established for weighted Korobov spaces with any smoothness and general weights. Moreover, a similar result has been shown for high order polynomial lattice point sets in weighted Sobolev spaces with smoothness $\alpha\geq 2$, getting rid of the necessity that the integrand should be periodic. Although each of $Q_{P_1},\ldots,Q_{P_r}$ is a linear non-adaptive algorithm, taking their median results in a non-linear algorithm. This non-linearity distinguishes it from the theoretically optimal deterministic integration methods. We justify the use of such a non-linear algorithm by observing that we are now working with randomized algorithms, where achieving a high confidence level is desirable because it is usually hard to check the actual worst-case error associated with every realization of randomized algorithms. Additionally, by utilizing the median trick, it becomes possible for the probability of ``failure'' in meeting a probabilistic error bound to decay exponentially fast toward $0$ as the number of trials $r$ increases, see \cite[Proposition~1]{KR19} and \cite[Theorem~2.5]{GL22} among some others. For the worst-case error analysis, Proposition~\ref{prop:meta} in this paper establishes a general median principle. Here we point out that taking the median of several randomized QMC estimates has been also studied quite recently in \cite{GLM22,HR22,PO21,PO22}.

In this paper, we push forward this novel research direction by extending some of the results shown in \cite{GL22}. Our primary aim is to establish a \emph{universality} of QMC-based integration methods in the sense that they do not need to be adjusted to the smoothness and the weights of non-periodic function spaces and yet exhibit the nearly optimal rate of convergence with high probability. To do so, we consider two special classes of QMC point sets, namely, linearly scrambled digital nets and (infinite-precision) polynomial lattice point sets, and cover three weighted function spaces with different smoothness classes simultaneously. Other than one function space which is exactly the same as the one studied in \cite{GL22} for high order polynomial lattice point sets, we consider a weighted Sobolev space of first order at one end, which contains functions with absolutely integrable mixed first derivatives, and a weighted space of infinitely many times differentiable functions at the other end. As related works, we refer to \cite{KN17,UU16,U17} which study the universality of (randomized) Frolov cubature \cite{Fr76} for unweighted function spaces with finite smoothness, and also to \cite{PO21} which shows a universality of one-dimensional scrambled nets in terms of smoothness of functions, covering both the finite and infinite smoothness classes. A part of the results in the latter paper has been extended to multi-dimensional scrambled nets quite recently in \cite{PO22}. However, a universality in terms of the weights has been largely missing in the literature.

The rest of this paper is organized as follows. In the next section, we introduce linearly scrambled digital nets and polynomial lattice point sets and give some new results on their dual properties. Then, in section~\ref{sec:theory}, we show a probabilistic worst-case error bound for each of three weighted function spaces and discuss its dependence on the dimension $s$. In section~\ref{sec:numerics}, we perform some numerical experiments to support our theoretical claims.

\paragraph{Notation}
Throughout this paper, we denote the set of positive integers by $\NN$ and write $\NN_0=\NN\cup \{0\}$. For $s\in \NN$, we use the shorthand $1{:}s$ to denote the set $\{1,\ldots,s\}$. We always assume that $b$ is a prime. We denote by $\Fb = \{0,1, \dots,b-1\}$ the field of $b$ elements, i.e., the simplest construction of the prime field of order $b$ by taking the integers modulo $b$, and by $\Fb^{n \times m}$ the set of $n \times m$ matrices over $\Fb$ for $m,n\in \NN$. For $k\in \NN_0$ having the $b$-adic expansion $k=\kappa_0+\kappa_1 b+\cdots$, where all but a finite number of $\kappa_i$’s are 0, we write $k(x)=\kappa_0+\kappa_1 x+\cdots\in \FF_b[x]$ and $\vec{k}=(\kappa_0,\kappa_1,\ldots)^{\top}\in \FF_b^{\NN}$.

\section{QMC point sets}
In this paper we study two special classes of QMC point sets, linear scrambled digital $(t,m,s)$-nets and (infinite-precision) polynomial lattice point sets. Both of them can be regarded as a class of digital nets, respectively, which we define as follows. 
\begin{definition}[digital net]
Let $m \in \NN$, $n \in \NN \cup \{\infty\}$ and $C_1, \dots, C_s \in \Fb^{n \times m}$. For each integer $0 \le k < b^m$, we denote its $b$-adic expansion by $k = \sum_{i=0}^{m-1} \kappa_i b^{i}$ with $\kappa_i \in \Fb$, and obtain the $k$-th point $\bsx_k\in [0,1]^s$ by
\begin{align}\label{eq:pt-generation}
\bsx_k := (\psi_n(\bsy_{k,1}), \dots, \psi_n(\bsy_{k,s})),
\end{align}
where we write $\bsy_{k,j} := C_j (\kappa_0, \dots, \kappa_{m-1})^\top \in \Fb^n$ for each $j$, and the map $\psi_n \colon \Fb^n \to [0,1]$ is defined by
\[ \psi_n((y_1, \dots, y_{n})^\top) := \sum_{i=1}^n \frac{y_{i}}{b^i}. \]
The $b^m$-element point set $\{\bsx_0,\ldots,\bsx_{b^m-1}\}$ constructed this way is called a digital net over $\Fb$ with generating matrices $C_1,\ldots,C_s$, which we denote by $P(C_1,\ldots,C_s)$.
\end{definition}

In this section, we first introduce the definitions of linear scrambled digital $(t,m,s)$-nets and polynomial lattice point sets, respectively, and then give some new results on their dual properties.
\subsection{Digital $(t,m,s)$-nets and linear scrambling}

The $t$-value is the central measure of how uniformly distributed a digital net is, which comes from the concept of $(t,m,s)$-nets defined in terms of the elementary intervals as follows.
\begin{definition}[elementary interval and $(t,m,s)$-net]\label{def:tms-net}
A $b$-adic $s$-dimensional elementary interval denotes an axis-parallel, right half-open box of the form
\[ \prod_{j=1}^{s}\left[ \frac{c_j}{b^{a_j}},\frac{c_j+1}{b^{a_j}}\right),\quad \text{with $a_j, c_j\in \NN_0$ and $0\leq c_j<b^{a_j}$.}\]
For $m\in \NN_0$, a $b^m$-element point set $P\subset [0,1)^s$ is called a $(t,m,s)$-net in base $b$ if every elementary interval of volume $b^{-m+t}$ contains exactly $b^t$ points of $P$.
\end{definition}

It is obvious from the definition that a smaller $t$-value ensures a finer uniformity of points and $t=0$ is the best possible. Now it is natural to call a point set $P$ a \emph{digital $(t,m,s)$-net in base $b$} if $P$ is a digital net over $\FF_b$ as well as a $(t,m,s)$-net in base $b$. Since any $b^m$-element point set or digital net over $\FF_b$ is an $(m,m,s)$-net in base $b$, the concept of $(t,m,s)$-net is only useful when $t<m$ \cite[Section~2.5]{DKS13}.

The least $t$-value of a given digital net can be determined in the following way. We refer to \cite[Theorem~4.52]{DP10book} for the proof.
\begin{lemma}\label{lem:t-value}
Let $m,n \in \NN$ and $C_1, \dots, C_s \in \Fb^{n \times m}$. The point set $P(C_1,\ldots,C_s)$ is a digital $(t,m,s)$-net in base $b$ with $t$ being the smallest non-negative integer such that, for any $a_1,\ldots,a_s\in \NN_0$ with $a_1+\cdots+a_s=m-t$,\\[5pt]
\indent the first $a_1$ rows of $C_1$,\\
\indent the first $a_2$ rows of $C_2$,\\
\indent $\vdots$\\
\indent and the first $a_s$ rows of $C_s$\\[5pt]
are linearly independent over $\FF_b$.
\end{lemma}

\begin{remark}
If $n=\infty$, the $t$-value determined by Lemma~\ref{lem:t-value} does not necessarily coincide with the one in the sense of Definition~\ref{def:tms-net}. This problem occurs because a digital net with generating matrices having infinitely many rows allows for two different representations of $b$-adic rationals, such as
\[ x=\frac{1}{b}\quad \text{and}\quad x=\frac{b-1}{b^2}+\frac{b-1}{b^3}+\cdots. \]
In such cases, we always consider the $t$-value in the sense of Lemma~\ref{lem:t-value}.
\end{remark}

Explicit constructions of digital $(t,m,s)$-nets with small $t$-value have been provided by Sobol' \cite{So67}, Niederreiter \cite{Ni88} and Tezuka \cite{Te93} as well as many others, see \cite[Chapter~4]{Ni92book} and \cite[Chapter~8]{DP10book}. As a key ingredient for the subsequent analysis, let us mention the result from \cite{Wa03} on the $t$-value of the Niederreiter sequence in base $b$: for any $m,s\in \NN$ and any non-empty subset $u\subseteq 1{:}s$, the projection of the first $b^m$ points from the $s$-dimensional Niederreiter sequence onto the $|u|$-dimensional unit cube $[0,1)^{|u|}$ is a digital $(t_u,m,|u|)$-net in base $b$ with
\begin{align}\label{eq:tvalue_Niederreiter}
    t_u\leq \sum_{j\in u}\left( \log_b j+\log_b\log_b(j+b)+1\right).
\end{align}
As explained in \cite[Example~2.20]{DKS13}, the Niederreiter sequence replaces the sequence of primitive polynomials over $\FF_2$, used in the Sobol' sequence construction, with a sequence of irreducible polynomials over $\FF_b$ with prime power $b$. As primitive polynomials are always irreducible, this replacement leads to a smaller bound on the $t$-value.

We now introduce the definition of linearly scrambled digital nets (without shift).
\begin{definition}[linearly scrambled digital net]\label{def:linear_scramble}
For $w,n \in \NN\cup \{\infty\}$, we define a set of non-singular lower triangular matrices over $\FF_b$, denoted by $\cL_{w,n}$, as
\[
\cL_{w,n} := \left\{L = (\ell_{i,j}) \in \Fb^{w \times n} \;\middle|\; \text{$\ell_{i,j}=0$ if $i<j$ and $\ell_{i,j}\neq 0$ if $i=j$} \right\}.
\]
For a digital net $P(C_1, \dots, C_s)$ with $C_1,\ldots,C_s\in \FF_b^{n\times m}$ and $L_1,\ldots,L_s\in \cL_{w,n}$, the corresponding linearly scrambled digital net (with precision $w$) is a digital net $P(L_1C_1,\ldots,L_sC_s)$. 
If each $L_j$ is independently and randomly chosen from $\cL_{w,n}$, we call $P(L_1C_1,\ldots,L_sC_s)$ a randomly linearly scrambled digital net over $\FF_b$. Here a random sampling from $\cL_{w,n}$ is done by choosing each element $\ell_{i,j}$ independently and randomly under the above conditions.
\end{definition}

\begin{remark}
To be more precise on random sampling from $\cL_{w,n}$ in Definition~\ref{def:linear_scramble},
let us denote the full index set by $U_{w,n}=\{(i,j)\mid 1\leq i\leq w, 1\leq j\leq n, i\leq j\}$ for given $w,n \in \NN\cup \{\infty\}$, and consider the Cartesian product 
\[
\cX_{w,n} :=
\prod_{\substack{(i,j)\in U_{w,n}\\ i=j}} (\FF_b\setminus \{0\}) \times \prod_{\substack{(i,j)\in U_{w,n}\\ i<j}} \FF_b,
\]
which is canonically isomorphic to $\cL_{w,n}.$
The Kolmogorov extension theorem ensures that,
even if either $w$ or $n$ is infinite, there exists a measure $\nu$ on $\cX_{w,n}$ induced by the uniform probability measures on $\FF_b\setminus \{0\}$ and $\FF_b$. Specifically, for every cylinder set
\[ E=\prod_{\substack{(i,j)\in V\\ i=j}}F_{ii}\times\prod_{\substack{(i,j)\in V\\i<j}}F_{ij}\times \prod_{\substack{(i,j)\in U_{w,n}\setminus V\\ i=j}} (\FF_b\setminus \{0\})\times \prod_{\substack{(i,j)\in U_{w,n}\setminus V\\ i<j}} \FF_b, \]
with a finite index subset $V\subset U_{w,n}$, $F_{ii}\subseteq \FF_b\setminus \{0\}$ and $F_{ij}\subseteq \FF_b$, it holds that 
\[ \nu(E)= \prod_{\substack{(i,j)\in V\\ i=j}}\frac{|F_{ii}|}{b-1}\times \prod_{\substack{(i,j)\in V\\i<j}}\frac{|F_{ij}|}{b}.\]
Subsequently, when computing the probability of an event or the expectation of functions with respect to $L_1,\ldots,L_s\in \cL_{w,n}$, we consider the $s$-fold product of this cylinder set measure.

In what follows, we focus on the case where $w=\infty$, which facilitates our theoretical analysis. In practice, representing infinite digit expansions on computers is quite hard, but, even with finite expansions, the accompanying truncation error diminishes as $w$ increases, as shown in \cite[Lemma~2.1]{MSM14} and \cite[Lemma~1 \& Corollary~3]{PO22}.
\end{remark}

\begin{remark}\label{rem:preserve_tvalue}
A deterministic linear scrambling as in Definition~\ref{def:linear_scramble} has been originally introduced by Tezuka \cite{Te94} to generalize Faure sequences. A random linear scrambling together with a random digital shift has been studied by Matou\v{s}ek \cite{Ma98} as a computationally more efficient alternative to the fully nested scrambling by Owen \cite{Ow95,Ow97a,Ow97b}. It is important that, no matter how we choose $L_1,\ldots,L_s\in \cL_{w,n}$ with $w\geq m$, the $t$-value of a linearly scrambled net $P(L_1C_1,\ldots,L_sC_s)$ remains the same as that of $P(C_1,\ldots,C_s)$, i.e., a linear scrambling preserves the $t$-value. In passing, we point out that further derandomization of random linear scrambling has been studied in \cite{Ow03,TF03}. For example, $i$-binomial scrambling developed in \cite{TF03} restricts the matrices in $\cL_{w,n}$ to those of Toeplitz type. Although we will not go into the details, the results of this paper also apply to such cases.
\end{remark}

\subsection{Polynomial lattice point sets}
As the other class of point sets we study in this paper, here we introduce polynomial lattice point sets.

\begin{definition}[polynomial lattice point set]\label{def:plr}
For $m,s\in \NN$ and $w\in \NN\cup \{\infty\}$ with $w\geq m$, let $p\in \FF_b[x]$ with $\deg(p)=m$ and $\bsg=(g_1,\ldots,g_s)\in (\FF_b[x])^s$ with $\deg(g_j)<m$. The polynomial lattice point set (of precision $w$) with modulus $p$ and generating vector $\bsg$ is a $b^m$-element point set
\[ P(p,\bsg,w)=\left\{ \bsx_k:=\nu_w\left( \frac{k(x)\bsg(x)}{p(x)}\right)\;\middle|\; 0\leq k<b^m\right\},\]
where $\nu_w: \FF_b((x^{-1}))\to [0,1]$ is applied component-wise to a vector and is given by
\[ \nu_w\left( \sum_{i=c}^{\infty}\frac{a_i}{x^i}\right)=\sum_{i=\max(1,c)}^{w}\frac{a_i}{b^i}.\]
When $w=\infty$, we simply write $P(p,\bsg)$ instead of $P(p,\bsg,\infty)$.
\end{definition}

Polynomial lattice point sets have been originally introduced by Niederreiter \cite{Ni92} and their relation to linear shift-register sequences has been observed in \cite{LL99}. We also refer to \cite{LL03,DP07,Pi12}. It is well-known that polynomial lattice point sets can be regarded as a special class of digital nets over $\FF_b$ as follows. We refer to \cite[Chapter~10]{DP10book} for the proof.

\begin{lemma}
For $m,s\in \NN$ and $w\in \NN\cup \{\infty\}$ with $w\geq m$, let $p\in \FF_b[x]$ with $\deg(p)=m$ and $\bsg=(g_1,\ldots,g_s)\in (\FF_b[x])^s$ with $\deg(g_j)<m$. The polynomial lattice point set $P(p,\bsg,w)$ is a digital net over $\FF_b$ with generating matrices $C_1,\ldots,C_s\in \FF_b^{w\times m}$, where each matrix $C_j=(c^{(j)}_{i,r})$ is of Hankel type:
\[ c^{(j)}_{i,r}=u^{(j)}_{i+r-1},\]
and $u^{(j)}_1,u^{(j)}_2,\ldots$ are the coefficients appearing in the infinite (periodic) Laurent series of the quotient
\[ \frac{g_j(x)}{p(x)}=\sum_{i=1}^{\infty}\frac{u^{(j)}_i}{x^i}\in \FF_b((x^{-1})). \]
\end{lemma}

\begin{remark}
Polynomial lattice point sets of infinite precision $w=\infty$ have been studied in the randomized setting carefully in \cite[Section~5]{DGS22}. We also refer to a relevant work \cite{GSY16} on digital nets with infinite digit expansions. In this paper, we focus only on the case $w=\infty$ in the theoretical analysis, but the result can be well-approximated by a large enough $w$. Note that, for $w=\infty$, QMC integration is given by
\begin{align*}
    Q_{P(p,\bsg)}(f) := \lim_{w\to \infty}Q_{P(p,\bsg,w)}(f)=\frac{1}{b^m}\sum_{k=0}^{b^m-1}f(\bsx_k-),
\end{align*}
where $f(\bsx-)=\lim_{\bsy\nearrow \bsx}f(\bsy)$ denotes the component-wise left limit. Therefore, if $f$ is not left continuous, we may have 
\[ Q_{P(p,\bsg)}(f)\neq Q_{P(p,\bsg,\infty)}(f)=\frac{1}{b^m}\sum_{k=0}^{b^m-1}f(\bsx_k),\]
whereas such a problem does not happen in this paper as we only deal with continuous functions.
\end{remark}

In the subsequent analysis, we restrict the modulus $p$ to be monic and irreducible. For $m\in \NN$, we write
\[ \PP_m:=\left\{ p\in \FF_b[x]\;\middle|\; \text{$\deg(p)=m$ and $p$ is monic and irreducible}\right\},\]
and
\[ \GG_m:=\left\{ g\in \FF_b[x]\;\middle|\; \text{$g\neq 0$ and $\deg(g)<m$} \right\}.\]
We shall always choose $p\in \PP_m$ and $\bsg \in (\GG_m)^s$. It is obvious that $|\GG_m|=b^m-1$ and also it is known from \cite[Lemma~4]{Po13} that $|\PP_m|\geq b^m/(2m).$

\subsection{Duality and some new results}

The error of QMC integration using digital nets is usually analyzed through the concept of dual nets, which are defined as follows.

\begin{definition}[Dual net]\label{def:dual_net}
Let $m \in \NN$, $n \in \NN \cup \{\infty\}$ and $C_1, \dots, C_s \in \Fb^{n \times m}$. The dual net of the digital net $P(C_1, \dots, C_s)$ is defined by
\[
P^\perp(C_1,\ldots,C_s) := \left\{\bsk \in \NN_0^s\;\middle|\; C_1^\top \tr_n(\vec{k}_1) + \cdots + C_s^\top \tr_n(\vec{k}_s) = \bszero \in \FF_b^m \right\},
\]
where we use the notation
\[ \tr_n(\vec{k})=(\kappa_0,\kappa_1,\ldots,\kappa_{n-1})^{\top}\in \FF_b^{n},\]
for $k=\kappa_0+\kappa_1b+\cdots\in \NN_0$, which stands for the truncated $b$-adic $n$-digit representation.
\end{definition}

In this subsection we give some new results on dual nets for linearly scrambled digital nets and polynomial lattice point sets, respectively.

\subsubsection{Linear scrambled digital nets}

We first describe an alternative way to Lemma~\ref{lem:t-value} for determining the least $t$-value of a digital net through its dual net. To do so, we need to introduce the Niederreiter--Rosenbloom--Tsfasman (NRT) weight.
\begin{definition}[NRT weight]\label{def:NRT_weight}
For $k\in \NN$, we denote its $b$-adic expansion by
\[ k=\kappa_1 b^{c_1-1}+\kappa_2 b^{c_2-1}+\cdots +\kappa_v b^{c_v-1},\]
with $\kappa_1,\ldots,\kappa_v\in \{1,\ldots,b-1\}$ and $c_1>c_2>\cdots>c_v>0$. The NRT weight $\mu_1: \NN_0\to \NN_0$ is defined by $\mu_1(k)=c_1$ and $\mu_1(0)=0$, or equivalently by $\mu_1(k)=\lceil \log_b(k+1)\rceil.$ Moreover, in the case of vectors in $\NN_0^s$, we define
\[ \mu_1(\bsk) = \sum_{j=1}^{s}\mu_1(k_j).\]
\end{definition}

Now the $t$-value and the dual net of a digital net can be associated with each other in the following way. We refer to \cite[Chapter~7]{DP10book} for the proof.
\begin{lemma}\label{lem:dual-t}
Let $m \in \NN$, $n \in \NN \cup \{\infty\}$ and $C_1, \dots, C_s \in \Fb^{n \times m}$. Define
\[ \mu_1(P^{\perp}(C_1,\ldots,C_s)):=\min_{\bsk\in P^{\perp}(C_1,\ldots,C_s)\setminus \{\bszero\}}\mu_1(\bsk). \]
Then the least $t$-value of the digital net $P(C_1,\ldots,C_s)$ equals 
\[ m-\mu_1(P^{\perp}(C_1,\ldots,C_s))+1.\]
\end{lemma}

In what follows, we give some new results on the dual net. Although we have already mentioned in Remark~\ref{rem:preserve_tvalue} that a linear scrambling preserves the $t$-value, we need a bit stronger result as follows.

\begin{lemma}
Let $k\in \NN$ and let us write $c := \mu_1(k)$. Then, for any $L\in \cL_{\infty,\infty}$, by letting $\ell\in \NN$ such that $\vec{\ell}=L^\top\vec{k}$, we always have $\mu_1(\ell) = c$. Moreover, for any $k'\in \NN$ satisfying $\mu_1(k') = c$ and any $n\geq c$, if $L$ is randomly sampled from $\cL_{\infty,n}$, we have
\[
\prob{ L^\top \vec{k} = \vec{k}' } = \frac{1}{(b-1)b^{c-1}}.
\]
\end{lemma}

\begin{proof}
The first claim is trivial from the definition of $\cL_{w,n}$. Noting that only the first $c$ rows of $L$ affect the result of the multiplication $L^\top \vec{k}$, and also that, for any choice of the first $c-1$ rows of $L$, there exists exactly one vector for the $c$-th row of $L$ such that $L^\top \vec{k} = \vec{k}'$ holds, the second claim follows immediately from the fact that the number of possible choices for the $c$-th row is $(b-1)b^{c-1}$.
\end{proof}

This lemma can be trivially extended to the multi-dimensional case as follows.

\begin{lemma}\label{lem:scrambling_prob}
For $\bsk = (k_1, \dots, k_s) \in \NN_0^s$, let us write $c_j := \mu_1(k_j)$ and $u := \{j \mid k_j \neq 0\}$. Then, for any $j$ and $L_j \in \cL_{\infty,\infty}$, by letting $\ell_j$ such that $\vec{\ell}_j=L_j^{\top}\vec{k}_j$, we have $\mu_1(\ell_j) = c_j$.
Moreover, for any $\bsk'=(k'_1,\ldots,k'_s)\in \NN_0^s$ with $\mu_1(k'_j) = c_j$ for all $j$ and any $n\geq \max_j c_j$, if $L_1,\dots,L_s$ are independently and randomly sampled from $\cL_{\infty,n}$, we have
\[
\prob{ L_j^\top \vec{k}_j = \vec{k}'_j \text{ for all $j$} }
= \left(\frac{b}{b-1}\right)^{|u|}b^{-(c_1+\cdots+c_s)}.
\]
\end{lemma}

The following result is a straightforward generalization of the recent result shown in \cite[Lemma~3.4]{GS22} by building up on \cite[Corollary~13.8]{DP10book}. We also refer to \cite[Lemma~2.2]{Sk06} for the previously known result.

\begin{lemma}\label{lem:gain_coeff}
Let $P$ be a digital net over $\FF_b$ of size $b^m$. For a non-empty $u \subseteq 1{:}s$, let $P_u:=\{\bsx_u\mid \bsx\in P\}$ be the projection of $P$ onto the $|u|$-dimensional unit cube and assume that $P_u$ is a digital $(t_u,m,|u|)$-net in base $b$. Then, for any $\bsc_u=(c_j)_{j\in u}\in \NN^{|u|}$, by writing $|\bsc_u|_1=\sum_{j\in u}c_j$, we have
\begin{align*}
& \left|\{ \bsk = (k_1, \dots, k_s) \in P^\perp \;\middle|\; \text{$\mu_1(k_j) = c_j$ if $j\in u$ and $k_j=0$ otherwise}\}\right| \\
& \quad \leq
\begin{cases}
0 & \text{if $|\bsc_u|_1 \le m-t_u$},\\
(b-1)^{|\bsc_u|_1-(m-t_u)} & \text{if $m-t_u < |\bsc_u|_1 \le m-t_u+|u|$},\\
(b-1)^{|u|}b^{|\bsc_u|_1-(m-t_u+|u|)} & \text{if $|\bsc_u|_1 > m-t_u+|u|$}.
\end{cases}
\end{align*}
In particular, we have
\begin{align*}
    & \left|\{ \bsk = (k_1, \dots, k_s) \in P^\perp \;\middle|\; \text{$\mu_1(k_j) = c_j$ if $j\in u$ and $k_j=0$ otherwise}\}\right| \\
    & \quad \leq \left(\frac{b-1}{b}\right)b^{|\bsc_u|_1-(m-t_u)}.
\end{align*}
\end{lemma}

We now prove what we need in the subsequent analysis. The result can be regarded as a generalization of the known result recently shown in \cite[Lemma~3 and Corollary~1]{PO22}.
\begin{lemma}\label{lem:prob_dual}
For $m\in \NN$, let $P=P(C_1,\ldots,C_s)$ be a digital net over $\FF_b$ of size $b^m$ with square generating matrices $C_1,\ldots,C_s \in \FF_b^{m \times m}$, and assume that, for every non-empty $u\subseteq 1{:}s$, the projection $P_u$ is a digital $(t_u,m,|u|)$-net in base $b$. Moreover, let $L_1,\dots,L_s$ be independently and randomly sampled from $\cL_{\infty,m}$. For any $\bsk = (k_1, \dots, k_s) \in \NN_0^s \setminus \{\bszero\}$, by writing $u=\{j\mid k_j\neq 0\}$, we have
\[
\prob{\bsk \in P^\perp(L_1C_1,\ldots,L_sC_s)}
\leq \left(\frac{b}{b-1}\right)^{|u|-1}b^{-m+t_u}, \]
if $k_j<b^m$ holds for all $j$. Otherwise, assuming that each $C_j$ is non-singular, we have
\[
\prob{\bsk \in P^\perp(L_1C_1,\ldots,L_sC_s) }
= \frac{1}{b^m}.
\]
\end{lemma}

\begin{proof}
First let us consider the case $k_j<b^m$ for all $j$. 
From Definition~\ref{def:dual_net}, we have an equivalence
\begin{align*}
& \bsk \in P^\perp(L_1C_1,\ldots,L_sC_s) \\
&\iff
(L_1C_1)^{\top}\vec{k}_1 + \dots + (L_sC_s)^{\top}\vec{k}_s = \bszero \in \FF_b^m\\
&\iff
\bsk' \in P^\perp(C_1,\ldots,C_s)\quad \text{with $k'_j\in \NN_0$ satisfying $\vec{k}'_j=L_j^{\top}\vec{k}_j$.}
\end{align*}
It follows from Lemmas~\ref{lem:scrambling_prob} and \ref{lem:gain_coeff} that
\begin{align*}
& \prob{\bsk \in P^\perp(L_1C_1,\ldots,L_sC_s) }\\
& = \sum_{\bsk' \in P^\perp(C_1,\ldots,C_s)} \prob{ L_j^\top \vec{k}_j = \vec{k}'_j \text{ for all $j$} } \\
& \leq
\left(\frac{b}{b-1}\right)^{|u|}b^{-|\bsc_u|_1} \\
& \quad \quad \times \left|\{ \bsk' \in P^\perp(C_1,\ldots,C_s) \;\middle|\; \text{$\mu_1(k'_j) = c_j$ if $j\in u$ and $k'_j=0$ otherwise}\}\right| \\
& \leq \left(\frac{b}{b-1}\right)^{|u|}b^{-|\bsc_u|_1}\left(\frac{b-1}{b}\right)b^{|\bsc_u|_1-(m-t_u)}=\left(\frac{b}{b-1}\right)^{|u|-1}b^{-(m-t_u)},
\end{align*}
which proves the first claim.

Let us move on to the second case in which there exists at least one index $j^*$ such that $k_{j^*}\geq b^m$. Because of the assumption that all $C_j$ are non-singular, we have an equivalence
\begin{align*}
\bsk \in P^\perp(L_1C_1,\ldots,L_sC_s)
&\iff
(L_1C_1)^{\top}\vec{k}_1 + \dots + (L_sC_s)^{\top}\vec{k}_s = \bszero \in \FF_b^m\\
&\iff
L_{j^*}^\top \vec{k}_{j^*} = - \sum_{j\neq j^*} (C_{j^*}^{\top})^{-1} (L_jC_j)^\top \vec{k}_j.
\end{align*}
Since $k_{j^*}\geq b^m$, there exists a unique integer $v>m$ such that $\vec{k}_{j^*,v} \neq 0$ and $\vec{k}_{j^*,v+1}=\vec{k}_{j^*,v+2}=\cdots=0$. Then we see that only the first $v$ rows of $L_{j^*}$ affect the result of the multiplication $L_{j^*}^\top \vec{k}_{j^*}$, and that, for any choice of the first $v-1$ rows of $L_{j^*}$, $L_{j^*}^\top \vec{k}_{j^*}$ is distributed uniformly on the set $\FF_b^m$ when the $v$-th row of $L_{j^*}$ is distributed uniformly on the set $\FF_b^m$. Hence we have
\[
\prob{L_{j^*}^\top \vec{k}_{j^*} = - \sum_{j\neq j^*} (C_{j^*}^{\top})^{-1} (L_jC_j)^\top \vec{k}_j} = \frac{1}{b^m},
\]
which proves the second claim.
\end{proof}

\begin{remark}
We note that, for the second claim of Lemma~\ref{lem:prob_dual}, we have assumed that each $C_j$ is non-singular. This implies that any one-dimensional projection of $P(C_1,\ldots,C_s)$ is a $(0,m,1)$-net in base $b$. Such a good property holds for Sobol' sequences and Niederreiter sequences, where a proper reordering of the rows of each matrix $C_j$ is necessary for the latter \cite{FL16}. Here reordering the rows does not change the $t$-value of nets shown in \eqref{eq:tvalue_Niederreiter}.
\end{remark} 

\subsubsection{Polynomial lattice point sets}
Here we show a result corresponding to polynomial lattice point sets. First, let us recall that the dual net of a polynomial lattice point set is expressed in a specific way as follows. We refer to \cite[Lemma~10.6]{DP10book} for the proof and also to \cite[Section~5.1]{DGS22} for a relevant result on the character property.
\begin{lemma}\label{lem:dual_plr}
For $m,s\in \NN$ and $w\in \NN\cup \{\infty\}$ with $w\geq m$, let $p\in \FF_b[x]$ with $\deg(p)=m$ and $\bsg\in (\FF_b[x])^s$ with $\deg(g_j)<m$. The dual net of the polynomial lattice point set $P(p,\bsg,w)$ is given by
\[ P^{\perp}(p,\bsg,w)=\left\{ \bsk\in \NN_0^s\;\middle|\;  \tr_w(\bsk(x))\cdot \bsg(x)\equiv 0 \pmod {p(x)}\right\},\]
where, extending the notation of Definition~\ref{def:dual_net}, we write
\[ tr_w(k(x))=\kappa_0+\kappa_1x+\cdots+\kappa_{w-1}x^{w-1},\]
for $k(x)=\kappa_0+\kappa_1x+\cdots\in \FF_b[x]$, which stands for the truncated polynomial over $\FF_b$ of degree less than $w$ and is applied component-wise to a vector. For $w=\infty$ we simply have $tr_w(k(x))=k(x)$ for any $k\in \NN_0$ and write $P^{\perp}(p,\bsg)$ instead of $P^{\perp}(p,\bsg,\infty)$.
\end{lemma}

Now the following result is what we need in the subsequent analysis. \begin{lemma}\label{lem:plr_walsh_average}
For $m,s\in \NN$, let $p$ and $\bsg$ be independently and randomly sampled from $\PP_m$ and $(\GG_m)^s$, respectively. For any $\bsk\in \NN_0^s\setminus \{\bszero\}$, it holds that
\[ \prob{\bsk\in P^{\perp}(p,\bsg)}\leq \frac{3\mu_1(\bsk)}{b^m-1}.\]
Moreover, if $k_j<b^m$ holds for all $j$, we have
\[ \prob{\bsk\in P^{\perp}(p,\bsg)}\leq \frac{1}{b^m-1}.\]
\end{lemma}
\begin{proof}
Let $\chi_{A}$ be the indicator function of an event $A$. It follows from Lemma~\ref{lem:dual_plr} that
\[  \prob{\bsk\in P^{\perp}(p,\bsg)} := \frac{1}{|\PP_m|}\sum_{p\in \PP_m}\frac{1}{|\GG_m|^s}\sum_{\bsg\in \GG_m^s}\chi_{\bsk\cdot \bsg\equiv 0 \pmod p},\]
where, with abuse of notation, $\bsk$ appearing on the right-hand side denotes the associated vector of polynomials in $\FF_b[x]$. If the modulus $p$ is a divisor of all the components of $\bsk$, on the one hand, the condition $\bsk\in P^{\perp}(p,\bsg)$ holds for any choice of $\bsg\in \GG_m^s$, which implies that 
\[ \frac{1}{|\GG_m|^s}\sum_{\bsg\in \GG_m^s}\chi_{\bsk\cdot \bsg\equiv 0 \pmod p}=1.\]

Assume, on the other hand, that there exists a non-empty subset $u\subseteq 1{:}s$ such that $k_j$ is not divisible by $p$ if and only if $j\in u$. Then the condition $\bsk\in P^{\perp}(p,\bsg)$ is equivalent to
$\bsk_u\cdot \bsg_u=\sum_{j\in u}k_jg_j\equiv 0 \pmod p.$ Here the cardinality of $u$ must be larger than 1, since, if $u=\{j\}$, the equivalent condition $k_jg_j\equiv 0 \pmod p$ contradicts the fact that $p$ is irreducible and $g_j\in \GG_m$. Thus, defining $\ell=\max_{j\in u}j$, we have $\ell\geq 2$, and the above condition $\sum_{j\in u}k_jg_j\equiv 0 \pmod p$ is further equivalent to
\[ k_{\ell}g_{\ell}\equiv -\bsk_{u\setminus \{\ell\}}\cdot \bsg_{u\setminus \{\ell\}}\equiv -\bsk_{1{:}\ell-1}\cdot \bsg_{1{:}\ell-1} \pmod p. \]
If $\bsk_{1{:}\ell-1}\cdot \bsg_{1{:}\ell-1}\equiv 0 \pmod p$ holds, no $g_{\ell}\in \GG_m$ satisfies this equation. Otherwise, there exists exactly one $g_{\ell}\in \GG_m$ which solves the equation. This leads to
\[ \frac{1}{|\GG_m|^s}\sum_{\bsg\in \GG_m^s}\chi_{\bsk\cdot \bsg\equiv 0 \pmod p}\leq \frac{1}{|\GG_m|}.\]

Now, since any polynomial $k\in \FF_b[x]$ has at most $\lfloor \deg(k)/m\rfloor $ prime divisors of a fixed degree $m$, by using $|\GG_m|=b^m-1$ and $|\PP_m|\geq b^m/(2m)$, we obtain
\begin{align*}
    \prob{\bsk\in P^{\perp}(p,\bsg)} & \leq \frac{1}{|\PP_m|}\sum_{\substack{p\in \PP_m\\ p\mid \bsk}}1+\frac{1}{|\PP_m|}\sum_{\substack{p\in \PP_m\\ p\nmid \bsk}}\frac{1}{|\GG_m|}\\
    & \leq \frac{1}{|\PP_m|}\frac{\max_{1\leq j\leq s}\deg(k_j)}{m}+\frac{1}{|\GG_m|}\\
    & \leq \frac{2\max_{1\leq j\leq s}\mu_1(k_j)}{b^m}+\frac{1}{b^m-1} \leq \frac{3\mu_1(\bsk)}{b^m-1},
\end{align*}
which proves the first statement of this lemma. The second statement follows immediately from the fact that there is no prime divisor of a fixed degree $m$ for any polynomial $k\in \FF_b[x]$ with $\deg(k)<m$.
\end{proof}

\section{Median QMC integration}\label{sec:theory}
Here we provide a basic framework for what we call \emph{median QMC integration}. Let $S$ be a collection of point sets in $[0,1]^s$. In the remainder of the paper, $S$ will be either a set of linearly scrambled digital nets
\begin{align}\label{eq:set1}
\left\{ P(L_1C_1,\ldots,L_sC_s)\;\middle|\; L_1,\ldots,L_s\in \cL_{\infty,n}\right\},
\end{align}
for some good generating matrices $C_1,\ldots,C_s\in \FF_b^{n\times m}$, or a set of polynomial lattices
\begin{align}\label{eq:set2}
\left\{ P(p,\bsg)\;\middle|\; p\in \PP_m, \bsg\in \GG_m^s\right\}. 
\end{align}
For an odd integer $r$, we draw point sets $P_1, \dots, P_r$ independently and randomly from the set $S$, and we approximate $I_s(f)$ by the median-of-means
\[  M_{r}(f) := \median\left(Q_{P_{1}}(f),\ldots,Q_{P_{r}}(f)\right). \]
Because of random selection for $P_1, \dots, P_r$, our estimate $M_{r}(f)$ is a random variable, which means that our median QMC integration is a randomized quadrature algorithm; for every realization of $P_1, \dots, P_r$, the resulting estimator has a certain worst-case error. In what follows, we are interested in a lower bound on the probability of a small worst-case error (less than or equal to $\epsilon$) with respect to independently and randomly chosen $P_1, \dots, P_r$. A study of our median QMC integration in terms of different randomized error criteria is left open for future research. We refer to \cite{KR19} as a relevant work in this research direction, which studies a probabilistic error guarantee of the form
\[ \sup_{\substack{f\in B\\ \|f\|_B\leq 1}}\prob{\left| A_n(f)-I_{s}(f)\right|> \epsilon} \leq \delta, \] 
for a randomized algorithm $A_n$ that uses $n$ function values.

\begin{remark}
Although we do not go into the details, the results shown in this section also apply to the collections of randomly digitally shifted point sets
\[ \left\{ P(L_1C_1,\ldots,L_sC_s)\oplus \Delta\;\middle|\; L_1,\ldots,L_s\in \cL_{\infty,n}, \Delta\sim U([0,1)^s)\right\},\] and 
\[ \left\{ P(p,\bsg)\oplus \Delta\;\middle|\; p\in \PP_m, \bsg\in \GG_m^s, \Delta\sim (U[0,1)^s)\right\}, \]
respectively. Here we write $P\oplus \Delta=\{\bsx\oplus \Delta\mid \bsx\in P\}$, and $\oplus$ denotes the digit-wise addition modulo $b$, that is, for $x=\xi_1/b+\xi_2/b^2+\cdots$ and $y=\eta_1/b+\eta_2/b^2+\cdots$, we define
\[ x\oplus y=\frac{\zeta_1}{b}+\frac{\zeta_2}{b^2}+\cdots, \quad \text{with $\zeta_i\equiv \xi_i+\eta_i \pmod b$,}\]
and apply component-wise to vectors. As mentioned in Remark~\ref{rem:preserve_tvalue}, the former set above is nothing but the set of randomly scrambled digital nets in the sense of Matou\v{s}ek \cite{Ma98}. Applying a random digital shift makes every individual estimator unbiased, whereas it is biased without a random digital shift. However, taking the median of $r$ repetitions makes the resulting final estimator biased regardless of whether a random digital shift is applied or not.
\end{remark}
\subsection{Basic result}
Let $B$ be a Banach space over $[0,1]^s$ with the norm $\|\cdot\|_B$. If a realization of $P_1,\ldots,P_r\in S$ is given, our median QMC integration $M_r$ is deterministic and the worst-case error is defined by
\[
  e^{\wor}(M_{r}; B) 
	:=  \sup_{\substack{f\in B\\ 
	 \|f\|_B\leq 1}} |M_{r}(f) -I_s(f)|.
\]
Then the question we are interested in is: \emph{What are the possible pairs $(\epsilon, \delta)$ that satisfy
    \[ \prob{e^{\wor}(M_{r}; B)\leq \epsilon}\geq 1-\delta, \]
if the probability is taken with respect to independently and randomly chosen $P_1, \dots, P_r$?} Regarding this question, we can establish the following ``meta" result.

\begin{proposition}\label{prop:meta}
Assume that there exists a pair $(\epsilon,\delta)$ such that $\delta\in (0,1/4)$ and 
\begin{align}\label{eq:existence}
    \prob{e^{\wor}(Q_{P}; B)\leq \epsilon} \geq 1-\delta
\end{align} 
holds if $P$ is randomly chosen from $S$. Then, if $P_1, \dots, P_r$ are independently and randomly chosen from $S$, we have
\[ \prob{e^{\wor}(M_{r}; B)\leq \epsilon}\geq 1-2^{r-1}\delta^{(r+1)/2}. \]
\end{proposition}

\begin{proof}
For each realization of $P_1,\ldots,P_r$, the corresponding worst-case error is bounded by
\begin{align*}
    e^{\wor}(M_{r}; B) & = \sup_{\substack{f\in B\\ 
	 \|f\|_B\leq 1}} \left|\median\left(Q_{P_{1}}(f)-I_s(f),\ldots,Q_{P_{r}}(f)-I_s(f)\right)\right| \\
	 & \leq \sup_{\substack{f\in B\\ 
	 \|f\|_B\leq 1}} \median\left(|Q_{P_{1}}(f)-I_s(f)|,\ldots,|Q_{P_{r}}(f)-I_s(f)|\right)\\
	 & \leq \sup_{\substack{f\in B\\ 
	 \|f\|_B\leq 1}} \median\left(\|f\|_B\, e^{\wor}(Q_{P_1}; B),\ldots,\|f\|_B\, e^{\wor}(Q_{P_r}; B)\right)\\
	 & = \median\left(e^{\wor}(Q_{P_1}; B),\ldots, e^{\wor}(Q_{P_r}; B)\right),
\end{align*}
where the first inequality follows from Jensen's inequality for medians, see \cite{Me05} and \cite[Lemma~2.6]{GL22}.

Following an argument similar to that in \cite[Proposition~2.1]{NP09}, the median of $e^{\wor}(Q_{P_1}; B),\ldots, e^{\wor}(Q_{P_r}; B)$ is larger than $\epsilon$ if and only if the number of point sets that satisfy $e^{\wor}(Q_{P_i}; B)>\epsilon$ is greater than or equal to $(r+1)/2$. Considering the Bernoulli scheme in which $e^{\wor}(Q_{P_i}; B)\leq \epsilon$ is interpreted as the ``success'' in the $i$-th trial, we obtain
\begin{align*}
    & \prob{\median\left(e^{\wor}(Q_{P_1}; B),\ldots, e^{\wor}(Q_{P_r}; B)\right)>\epsilon} \\
    & = \sum_{i=(r+1)/2}^{r}\binom{r}{i}\left(\prob{e^{\wor}(Q_{P}; B)>\epsilon}\right)^i\left(1-\prob{e^{\wor}(Q_{P}; B)>\epsilon}\right)^{r-i} \\
    & \leq \sum_{i=(r+1)/2}^{r}\binom{r}{i}\delta^i \leq \delta^{(r+1)/2}\sum_{i=(r+1)/2}^{r}\binom{r}{i}=2^{r-1}\delta^{(r+1)/2}.
\end{align*}
By combining this with the earlier bound on $e^{\wor}(M_{r}; B)$, we complete the proof.
\end{proof}

As already pointed out in \cite[Remark~2.7]{GL22}, since we assume $0<\delta<1/4$, the probability given in Proposition~\ref{prop:meta} converges to 1 exponentially fast as $r$ increases. Therefore it is clear that, with the help of taking the median, increasing $r$ amplifies the probability of success and what we need to do is to prove that most of the point sets in $S$ are ``good" in the sense of \eqref{eq:existence}. To prove such an existence, probably the most straightforward way is to apply Markov inequality which ensures that $\epsilon$ can be given by
\begin{align}\label{eq:markov}
    \epsilon=\frac{1}{\delta}\EE_{P\in S}\left[ e^{\wor}(Q_{P}; B)\right],
\end{align}
so that it suffices to evaluate the expectation of the worst-case error with respect to $P\in S$. However, it is generally hard to show a rate of convergence better than $N^{-1}$ by \eqref{eq:markov}. In such situations, together with Markov inequality, we shall use Jensen's inequality
\begin{align}\label{eq:jensen}
    \phi\left( \sum_{n}a_n\right)\leq \sum_{n}\phi( a_n),
\end{align} 
which holds for any sequence of non-negative real numbers $(a_n)_n$ and concave function $\phi: [0,\infty)\to [0,\infty)$, see \cite[Section~2.3]{DHP15}. Then $\epsilon$ can be given by
\begin{align}\label{eq:epsilon}
    \epsilon=\inf_{\phi\in \Phi}\phi^{-1}\left(\frac{1}{\delta}\EE_{P\in S}\left[ \phi(e^{\wor}(Q_{P}; B))\right] \right),
\end{align}
for a family $\Phi$ of concave functions $\phi.$

\subsection{Universality}

In order to demonstrate the universality of our median QMC integration, we prove upper bounds on the expected worst-case error (in the sense of \eqref{eq:markov} or \eqref{eq:epsilon}) in function spaces with different smoothness classes, which constitutes the main contribution of this paper. We consider three weighted function spaces with different smoothness classes as our target Banach space. By following the seminal work by Sloan and Wo\'{z}niakowski \cite{SW98}, we model the relative importance of a subset of variables $\bsx_u=(x_j)_{j\in u}$ by introducing a set of weight parameters $\bsgamma=(\gamma_u)_{u\subseteq 1{:}s}$ with $\gamma_u\in [0,1]$ for the first two spaces with finite smoothness. For the third one, we follow \cite{Su17,DGSY17} and introduce a sequence $u_1\geq u_2\geq \dots >0$ to define a weighted space of infinitely many times differentiable functions. For each function space, we derive the possible values of the quantity $\epsilon$ for a given $\delta\in (0,1),$ and then derive sufficient conditions on the weights for the obtained values to be bounded independently of the dimension. At the end of each sub-subsection, we include brief comments for comparison on what is known in the literature.

\subsubsection{Weighted Sobolev space of first order}\label{subsubsec:main1}

The first function space, studied for instance in \cite{SW98,DLP05}, requires the minimum differentiability for integrands among the three spaces we consider.

\begin{definition}[weighted Sobolev space of first order]\label{def:Sob_space_first}
For a set of weights $\bsgamma$, the weighted Sobolev space of first order, denoted by $\cF_{s,1,\bsgamma}^{\Sob}$, is a Banach space with the norm
\[ \|f\|_{s,1,\bsgamma}^{\Sob}:=\sum_{u\subseteq 1{:}s}\gamma_u^{-1}\int_{[0,1]^{|u|}}\left|\frac{\partial^{|u|}}{\partial \bsx_u}f(\bsx_u,\bsone)\right|\rd \bsx_u,\]
where $(\bsx_u,\bsone)$ denotes the vector $\bsy\in [0,1]^s$ such that $y_j=x_j$ if $j\in u$ and $y_j=1$ otherwise. For any subset $u\subseteq 1{:}s$ with $\gamma_u=0$, we assume that the corresponding integral over $[0,1]^{|u|}$ equals $0$ and we set $0/0=0$.
\end{definition}

It is well known that, for any point set $P$, the worst-case error in $\cF_{s,1,\bsgamma}^{\Sob}$ is bounded above by the weighted star discrepancy:
\[ e^{\wor}(Q_P; \cF_{s,1,\bsgamma}^{\Sob})\leq D^{*}_{\bsgamma}(P):=\sup_{\bsy\in [0,1]^s}\max_{\emptyset\neq u\subseteq 1{:}s} \gamma_u \left| \Delta_{P}(\bsy_u,\bsone)\right|,\]
where $\Delta_{P}$ is the local discrepancy function defined by
\[ \Delta_{P}(\bsy)=\frac{1}{|P|}\sum_{\bsx\in P}\bsone_{\bsx\in [\bszero,\bsy)}-\prod_{j=1}^{s}y_j,\]
with $[\bszero,\bsy)=[0,y_1)\times \cdots \times [0,y_s)$ being the anchored axis-parallel box, see \cite[Section~3]{SW98}. According to \cite[Corollary~10.16]{DP10book}, if $P$ is a digital net over $\FF_b$ with square generating matrices $C_1,\ldots,C_s\in \FF_b^{m\times m}$, the weighted star discrepancy is bounded above by
\[ D^{*}_{\bsgamma}(P)\leq \sum_{\emptyset\neq u\subseteq 1{:}s}\gamma_u \Biggl[ 1-\left( 1-\frac{1}{b^m}\right)^{|u|}+\sum_{\substack{\bsk_u\in P_{u,0}^{\perp}\\ \forall j\in u, k_j<b^m}}\prod_{j\in u}\tilde{r}_{b}(k_j)\Biggr],\]
where $\tilde{r}_{b}: \NN_0\to \RR_{\geq 0}$ is defined by
\[ \tilde{r}_{b}(k)=\begin{cases} 1 & \text{if $k=0$,}\\ b^{-a}(\sin(\pi \kappa_{a-1}/b))^{-2} & \text{if $k=\kappa_0+\kappa_1 b+\cdots + \kappa_{a-1}b^{a-1}$ with $\kappa_{a-1}\neq 0$,}\end{cases}\]
and, for a non-empty subset $u\subseteq 1{:}s$, we write
\[ P_{u,0}^{\perp}=\left\{ \bsk_u\in \NN_0^{|u|}\setminus \{\bszero\} \;\middle|\; (\bsk_u,\bszero)\in P^{\perp}\right\}.\]

First let $S$ be given by \eqref{eq:set1} with good square generating matrices. The following result holds for this first function space.
\begin{theorem}\label{thm:main1_net}
For a set of weights $\bsgamma$, let $B=\cF_{s,1,\bsgamma}^{\Sob}$. For $m\in \NN$, let $S$ be a set of linear scrambled digital nets \eqref{eq:set1} with $C_1,\ldots,C_s\in \FF_b^{m\times m}$ all non-singular and satisfying \eqref{eq:tvalue_Niederreiter}. In Proposition~\ref{prop:meta}, the assumption \eqref{eq:existence} holds for any $\delta\in (0,1)$ and 
\[ \epsilon = \frac{3}{\delta b^m}\sum_{\emptyset\neq u\subseteq 1{:}s}\gamma_u \prod_{j\in u}\left[1+ m\frac{b(b+1)}{3}j\log_b(j+b)\right]. \]
\end{theorem}
\begin{proof}
We use \eqref{eq:markov} to prove the theorem. In this proof, we denote the left upper $m\times m$ submatrices of $L_1C_1,\ldots,L_sC_s$ by $(L_1C_1)^{(m)},\ldots,(L_sC_s)^{(m)}$, respectively. It follows from the linearity of expectation that
\begin{align}
    & \EE\left[ e^{\wor}(Q_{P(L_1C_1,\ldots,L_sC_s)}; \cF_{s,1,\bsgamma}^{\Sob})\right] \notag \\
    & \leq \EE\left[ D^{*}_{\bsgamma}(P(L_1C_1,\ldots,L_sC_s))\right] \notag \\
    & = \EE\left[ D^{*}_{\bsgamma}(P((L_1C_1)^{(m)},\ldots,(L_sC_s)^{(m)}))\right] \notag \\
    & \quad + \EE\left[ D^{*}_{\bsgamma}(P(L_1C_1,\ldots,L_sC_s))-D^{*}_{\bsgamma}(P((L_1C_1)^{(m)},\ldots,(L_sC_s)^{(m)}))\right] \notag \\
    & \leq \sum_{\emptyset\neq u\subseteq 1{:}s}\gamma_u \EE\Biggl[ 1-\left( 1-\frac{1}{b^m}\right)^{|u|}+\sum_{\substack{\bsk_u\in P_{u,0}^{\perp}((L_1C_1)^{(m)},\ldots,(L_sC_s)^{(m)})\\ \forall j\in u, k_j<b^m}}\prod_{j\in u}\tilde{r}_{b}(k_j)\Biggr] \notag \\
    & \quad + \EE\left[ D^{*}_{\bsgamma}(P(L_1C_1,\ldots,L_sC_s))-D^{*}_{\bsgamma}(P((L_1C_1)^{(m)},\ldots,(L_sC_s)^{(m)}))\right]. \label{eq:disc_bound}
\end{align}

Regarding the expectations for the first term on the rightmost side above, since we have 
\begin{align*}
    & P^{\perp}((L_1C_1)^{(m)},\ldots,(L_sC_s)^{(m)}) \cap \{0,\ldots,b^m-1\}^s \\
    & = P^{\perp}(L_1C_1,\ldots,L_sC_s) \cap \{0,\ldots,b^m-1\}^s,
\end{align*}
we can apply the first statement of Lemma~\ref{lem:prob_dual} to obtain
\begin{align*}
    & \EE\Biggl[\; \sum_{\substack{\bsk_u\in P_{u,0}^{\perp}((L_1C_1)^{(m)},\ldots,(L_sC_s)^{(m)})\\ \forall j\in u, k_j<b^m}}\prod_{j\in u}\tilde{r}_{b}(k_j)\Biggr] \\
    & = \sum_{\substack{\bsk_u\in \NN_0^{|u|}\setminus \{\bszero\}\\ \forall j\in u, k_j<b^m}}\prob{(\bsk_u,\bszero) \in P^\perp(L_1C_1,\ldots,L_sC_s) } \prod_{j\in u}\tilde{r}_{b}(k_j)\\
    & = \sum_{\emptyset \neq v\subseteq u}\sum_{\substack{\bsk_v\in \NN^{|v|}\\ \forall j\in v, k_j<b^m}}\prob{(\bsk_v,\bszero) \in P^\perp(L_1C_1,\ldots,L_sC_s)} \prod_{j\in v}\tilde{r}_{b}(k_j) \\
    & \leq \sum_{\emptyset \neq v\subseteq u}\left(\frac{b}{b-1}\right)^{|v|-1}b^{-m+t_v}\sum_{\substack{\bsk_v\in \NN^{|v|}\\ \forall j\in v, k_j<b^m}}\prod_{j\in v}\tilde{r}_{b}(k_j)\\
    & = \sum_{\emptyset \neq v\subseteq u}\left(\frac{b}{b-1}\right)^{|v|-1}b^{-m+t_v}\left( m\frac{b^2-1}{3b} \right)^{|v|} \leq \frac{1}{b^m}\sum_{\emptyset \neq v\subseteq u}b^{t_v}\left( m\frac{b+1}{3} \right)^{|v|},
\end{align*}
where we used the result on a sum of $\tilde{r}_{b}(k)$, shown in \cite[Lemma~2.2]{DLP05}, for the last equality. By using a bound \eqref{eq:tvalue_Niederreiter} on the $t$-value of the projected point set, we have
\begin{align*}
    & \EE\left[ D^{*}_{\bsgamma}(P((L_1C_1)^{(m)},\ldots,(L_sC_s)^{(m)}))\right]  \\
    & \leq \sum_{\emptyset\neq u\subseteq 1{:}s}\gamma_u \left[ 1-\left( 1-\frac{1}{b^m}\right)^{|u|}+\frac{1}{b^m}\sum_{\emptyset \neq v\subseteq u}b^{t_v}\left( m\frac{b+1}{3} \right)^{|v|}\right] \\
    & \leq \sum_{\emptyset\neq u\subseteq 1{:}s}\gamma_u \left[ 1-\frac{1}{b^m}-\left( 1-\frac{1}{b^m}\right)^{|u|}+\frac{1}{b^m}\sum_{v\subseteq u}\left( m\frac{b(b+1)}{3} \right)^{|v|}\prod_{j\in v} j\log_b(j+b)\right]\\
    & = \sum_{\emptyset\neq u\subseteq 1{:}s}\gamma_u \left[ 1-\frac{1}{b^m}-\left( 1-\frac{1}{b^m}\right)^{|u|}+\frac{1}{b^m}\prod_{j\in u}\left[1+ m\frac{b(b+1)}{3}j\log_b(j+b)\right]\right] \\
    & \leq \frac{1}{b^m}\sum_{\emptyset\neq u\subseteq 1{:}s}\gamma_u \left[ |u|-1+\prod_{j\in u}\left[1+ m\frac{b(b+1)}{3}j\log_b(j+b)\right]\right] \\
    & \leq \frac{2}{b^m}\sum_{\emptyset\neq u\subseteq 1{:}s}\gamma_u \prod_{j\in u}\left[1+ m\frac{b(b+1)}{3}j\log_b(j+b)\right],
\end{align*}
where the third inequality follows from the fact that $1-\frac{1}{b^m}-\left( 1-\frac{1}{b^m}\right)^{|u|}\leq \frac{|u|-1}{b^m}$ holds for any non-empty $u\subseteq 1{:}d$, which itself can be proven by an induction on the cardinality $|u|$, and the last inequality is obtained by noticing that $$|u|-1\leq 2^{|u|}=\prod_{j\in u}(1+1)\leq \prod_{j\in u}\left[1+ m\frac{b(b+1)}{3}j\log_b(j+b)\right].$$

For the second term of \eqref{eq:disc_bound}, by denoting the $k$-th points of $P(L_1C_1,\ldots,L_sC_s)$ and $P((L_1C_1)^{(m)},\ldots,(L_sC_s)^{(m)})$ by $\bsx_k$ and $\tilde{\bsx}_k$, respectively, we have 
\[ \tilde{x}_{k,j}-x_{k,j}\leq \frac{1}{b^m},\]
for all $1\leq j\leq s$, so that, by applying \cite[Theorem~3.15]{DP10book}, we obtain
\begin{align*}
    & \EE\left[ D^{*}_{\bsgamma}(P(L_1C_1,\ldots,L_sC_s))-D^{*}_{\bsgamma}(P((L_1C_1)^{(m)},\ldots,(L_sC_s)^{(m)}))\right] \\
    & \leq \frac{1}{b^m}\sum_{\emptyset\neq u\subseteq 1{:}s}\gamma_u |u|\leq \frac{1}{b^m}\sum_{\emptyset\neq u\subseteq 1{:}s}\gamma_u \prod_{j\in u}\left[1+ m\frac{b(b+1)}{3}j\log_b(j+b)\right].
\end{align*}
Now the result of the theorem follows from \eqref{eq:markov}.
\end{proof}

Using \cite[Lemma~3]{HN03}, we can derive a condition under which the worst-case error is bounded independently of the dimension $s$ in the case of product weights, i.e., the case where, for any subset $u\subseteq 1{:}s$, the weight $\gamma_u$ is given by $\prod_{j\in u}\gamma_j$ with a one-dimensional sequence $\gamma_1,\gamma_2,\dots.$

\begin{corollary}\label{rem:main1_net}
For product weights satisfying $\sum_{j=1}^{\infty}\gamma_j j\log_b(j+b)<\infty$, the probabilistic worst-case error bound in Theorem~\ref{thm:main1_net} is bounded independently of $s$ by
\[ C_{(\gamma_1,\gamma_2,\ldots),\lambda} b^{-(1-\lambda) m},\]
where $\lambda>0$ is arbitrarily small, and $C_{(\gamma_1,\gamma_2,\ldots),\lambda}$ is a positive constant that satisfies $\lim_{\lambda\to 0^+}C_{(\gamma_1,\gamma_2,\ldots),\lambda}=\infty$.
\end{corollary}

Next let $S$ be given by \eqref{eq:set2}, for which the associated result is as follows. By replacing the first statement of Lemma~\ref{lem:prob_dual} with the second statement of Lemma~\ref{lem:plr_walsh_average}, the proof is almost identical to that of Theorem~\ref{thm:main1_net}, so we omit it.
\begin{theorem}\label{thm:main1_plr}
For a set of weights $\bsgamma$, let $B=\cF_{s,1,\bsgamma}^{\Sob}$. For $m\in \NN$, let $S$ be a set of polynomial lattices \eqref{eq:set2}. In Proposition~\ref{prop:meta}, the assumption \eqref{eq:existence} holds for any $\delta\in (0,1)$ and 
\[ \epsilon=\frac{3}{\delta (b^m-1)}\sum_{\emptyset\neq u\subseteq 1{:}s}\gamma_u \left( 1+m\frac{b^2-1}{3} \right)^{|u|}.\]
\end{theorem}

Following an argument similar to that used in \cite[Corollaries 5.45 \& 10.30]{DP10book}, we can derive a weaker condition than that presented above in Corollary~\ref{rem:main1_net} under which the worst-case error is bounded independently of the dimension $s$.

\begin{corollary}\label{rem:main1_plr}
For product weights satisfying $\sum_{j=1}^{\infty}\gamma_j<\infty$, the probabilistic worst-case error bound given in Theorem~\ref{thm:main1_plr} is bounded independently of $s$ with a form similar to that in Corollary~\ref{rem:main1_net}.
\end{corollary}

For the unweighted Sobolev space with the slightly different norm
\[ \left(\sum_{u\subseteq 1{:}s}\int_{[0,1]^{|u|}}\left|\frac{\partial^{|u|}}{\partial \bsx_u}f(\bsx_u,\bszero)\right|^2\rd \bsx_u\right)^{1/2},\]
a lower bound on the worst-case error of order $(\log N)^{(s-1)/2}/N$ is proven, for instance, in \cite[Theorem~2.5]{DHP15}, which also applies to $\cF_{s,1,\bsgamma}^{\Sob}$ with leading weight $\gamma_{\{1,\ldots,s\}}\neq 0$. Recall that the number of points in Theorems \ref{thm:main1_net} and \ref{thm:main1_plr} is $N=b^m$. Then, we see that, with fixed values of $\delta$ and $b$, our obtained probabilistic worst-case error bounds are of order $(\log N)^s/N$, which is nearly optimal. Regarding the dimension-independence of the weighted star discrepancy, it was proven in \cite{DLP05} that the sufficient condition $\sum_{j=1}^{\infty}\gamma_j<\infty$ holds for polynomial lattices with finite precision constructed by component-by-component algorithms. In addition, a non-constructive existence result was shown in \cite{Ai14} that a weaker condition $\sum_{j=1}^{\infty}\exp(-c\gamma_j^{-2})<\infty$ is sufficient for dimension-independence, but with a slower decay rate in terms of $N$.

\subsubsection{Weighted Sobolev space of high order}\label{subsubsec:main2}
The second function space allows for higher order differentiability than the first space, which has been introduced in the context of partial differential equations with random coefficients \cite{DKLNS14}. 

\begin{definition}[weighted Sobolev space of high order]\label{def:Sob_space_high}
For $\alpha\in \NN$, $\alpha\geq 2$, $1\leq q\leq \infty$ and a set of weights $\bsgamma$, the weighted Sobolev space of order $\alpha$, denoted by $\cF^{\Sob}_{s,\alpha,\bsgamma,q}$, is a Banach space with the norm
\begin{align*}
    & \|f\|^{\Sob}_{s,\alpha,\bsgamma,q} := \\
    & \quad \sup_{u\subseteq 1{:}s}\gamma_u^{-1}\left(\sum_{v\subseteq u}\sum_{\bstau_{u\setminus v}\in \{1,\ldots,\alpha\}^{|u\setminus v|}}\int_{[0,1)^{|v|}}\left| \int_{[0,1)^{s-|v|}} f^{(\bstau_{u\setminus v},\bsalpha_v,\bszero)}(\bsx) \rd \bsx_{-v}\right|^{q}\rd \bsx_v\right)^{1/q},
\end{align*}
where $(\bstau_{u\setminus v},\bsalpha_v,\bszero)$ denotes the vector $\bsh\in \NN_0^s$ such that $h_j=\tau_j$ if $j\in u\setminus v$, $h_j=\alpha$ if $j\in v$, and $h_j=0$ otherwise, and $f^{(\bstau_{u\setminus v},\bsalpha_v,\bszero)}$ denotes the mixed derivative of order $(\bstau_{u\setminus v},\bsalpha_v,\bszero)$ of $f$. For any subset $u\subseteq 1{:}s$ with $\gamma_u=0$, we assume that all the corresponding integrals over $[0,1]^{|v|}$ with $v\subseteq u$ equal $0$ and we set $0/0=0$.
\end{definition}

For a digital net $P$, the worst-case error in $\cF_{s,\alpha,\bsgamma,q}^{\Sob}$ has been analyzed through Walsh analysis for smooth functions \cite{Di08,Di09}, where the Dick weight, a generalization of the NRT weight, plays a central role. 
\begin{definition}[Dick weight]\label{def:dick_weight}
Let $\alpha\in \NN$, $\alpha\geq 2$. For $k\in \NN$, we denote its $b$-adic expansion by
\[ k=\kappa_1 b^{c_1-1}+\kappa_2 b^{c_2-1}+\cdots +\kappa_v b^{c_v-1},\]
with $\kappa_1,\ldots,\kappa_v\in \{1,\ldots,b-1\}$ and $c_1>c_2>\cdots>c_v>0$. The Dick weight $\mu_\alpha: \NN_0\to \NN_0$ is defined by $\mu_\alpha(0)=0$ and
\[ \mu_\alpha(k)=\sum_{i=1}^{\min(\alpha,v)}c_i.\]
Moreover, in the case of vectors in $\NN_0^s$, we define
\[ \mu_\alpha(\bsk) = \sum_{j=1}^{s}\mu_\alpha(k_j).\]
\end{definition}

According to \cite[Theorem~3.5]{DKLNS14}, for any digital net $P$, the worst-case error in $\cF_{s,\alpha,\bsgamma,q}^{\Sob}$ is bounded independently of the parameter $q$ by
\begin{align}\label{eq:error_bound_high_order}
   e^{\wor}(Q_P; \cF_{s,\alpha,\bsgamma,q}^{\Sob})\leq \sum_{\emptyset\neq u\subseteq 1{:}s} \gamma_u C_{\alpha}^{|u|}\sum_{\bsk_u\in P_u^{\perp}}b^{-\mu_{\alpha}(\bsk_u)}, 
\end{align}
where
\begin{align*}
    C_{\alpha} = \left(1+\frac{1}{b}+\frac{1}{b(b+1)} \right)^{\alpha-2}\left(3+\frac{2}{b}+\frac{2b+1}{b-1} \right) \max\left( \frac{2}{(2\sin\frac{\pi}{b})^{\alpha}}\max_{1\leq \tau< \alpha}\frac{1}{(2\sin\frac{\pi}{b})^{\tau}}\right),
\end{align*}
and, for a non-empty subset $u\subseteq 1{:}s$, we write
\[ P_u^{\perp}=\left\{ \bsk_u\in \NN^{|u|} \;\middle|\; (\bsk_u,\bszero)\in P^{\perp}\right\}.\]
Here we point out a slight difference from $P_{u,0}^{\perp}$, which has been introduced and used in section~\ref{subsubsec:main1}.

First let $S$ be given by \eqref{eq:set1} with good square generating matrices. The following result holds for the second function space.
\begin{theorem}\label{thm:main2_net}
For $\alpha\in \NN$, $\alpha\geq 2$, $1\leq q\leq \infty$ and a set of weights $\bsgamma$, let $B=\cF_{s,\alpha,\bsgamma,q}^{\Sob}$. For $m\in \NN$, let $S$ be a set of linearly scrambled digital nets \eqref{eq:set1} with $C_1,\ldots,C_s\in \FF_b^{m\times m}$ all non-singular and satisfying \eqref{eq:tvalue_Niederreiter}. In Proposition~\ref{prop:meta}, the assumption \eqref{eq:existence} holds for any $\delta\in (0,1)$ and 
\[ \epsilon = \inf_{\lambda\in (1/\alpha,1]}\left(\frac{1}{\delta b^m}\sum_{\emptyset\neq u\subseteq 1{:}s} \gamma^{\lambda}_u \left(\frac{b^2}{b-1}\right)^{|u|}C_{\alpha}^{\lambda |u|}A_{\alpha,\lambda}^{|u|}\prod_{j\in u} j\log_b(j+b)\right)^{1/\lambda},\]
with
\[ A_{\alpha,\lambda} :=\sum_{k=1}^{\infty}b^{-\lambda\mu_\alpha(k)} = \sum_{\tau=1}^{\alpha-1}\prod_{i=1}^{\tau}\frac{b-1}{b^{\lambda i}-1}+\frac{b^{\lambda\alpha}-1}{b^{\lambda\alpha}-b}\prod_{i=1}^{\alpha}\frac{b-1}{b^{\lambda i}-1}.\]
\end{theorem}

\begin{proof}
We use \eqref{eq:epsilon} to prove the theorem. For a family of concave functions, let $\phi(x)=x^{\lambda}$ with $0<\lambda\leq 1$. By applying Jensen's inequality \eqref{eq:jensen}, the linearity of expectation and Lemma~\ref{lem:prob_dual}, we have
\begin{align*}
    & \EE\left[ \phi\left(e^{\wor}(Q_{P(L_1C_1,\ldots,L_sC_s)}; \cF_{s,\alpha,\bsgamma,q}^{\Sob})\right)\right] \\
    & \leq \EE\left[ \sum_{\emptyset\neq u\subseteq 1{:}s} \gamma^{\lambda}_u C_{\alpha}^{\lambda |u|}\sum_{\bsk_u\in P_u^{\perp}(L_1C_1,\ldots,L_sC_s)}b^{-\lambda \mu_{\alpha}(\bsk_u)}\right] \\
    & = \sum_{\emptyset\neq u\subseteq 1{:}s} \gamma^{\lambda}_u C_{\alpha}^{\lambda |u|}\EE\left[\sum_{\bsk_u\in P_u^{\perp}(L_1C_1,\ldots,L_sC_s)}b^{-\lambda \mu_{\alpha}(\bsk_u)}\right] \\
    & = \sum_{\emptyset\neq u\subseteq 1{:}s} \gamma^{\lambda}_u C_{\alpha}^{\lambda |u|}\sum_{\bsk_u\in \NN^{|u|}}b^{-\lambda \mu_{\alpha}(\bsk_u)} \prob{\bsk_u \in P_u^\perp(L_1C_1,\ldots,L_sC_s)}\\
    & \leq \frac{1}{b^m}\sum_{\emptyset\neq u\subseteq 1{:}s} \gamma^{\lambda}_u C_{\alpha}^{\lambda |u|} \Biggl[\left(\frac{b}{b-1}\right)^{|u|-1}b^{t_u}\sum_{\substack{\bsk_u\in \NN^{|u|}\\ \forall j\in u, k_j<b^m}}b^{-\lambda \mu_{\alpha}(\bsk_u)}+\sum_{\substack{\bsk_u\in \NN^{|u|}\\ \exists j\in u, k_j\geq b^m}}b^{-\lambda \mu_{\alpha}(\bsk_u)}\Biggr]\\
    & \leq \frac{1}{b^m}\sum_{\emptyset\neq u\subseteq 1{:}s} \gamma^{\lambda}_u C_{\alpha}^{\lambda |u|}\left(\frac{b}{b-1}\right)^{|u|}b^{t_u} \sum_{\bsk_u\in \NN^{|u|}}b^{-\lambda \mu_{\alpha}(\bsk_u)}.
\end{align*}
Here we know from \cite[Lemma~7]{G16} that the inner sum over $\bsk_u$ is finite for any $\lambda\in (1/\alpha,1]$ and equals $A_{\alpha,\lambda}^{|u|}$. By using a bound \eqref{eq:tvalue_Niederreiter} on the $t$-value of the projected point set, we have
\begin{align*}
    & \EE\left[ \phi\left(e^{\wor}(Q_{P(L_1C_1,\ldots,L_sC_s)}; \cF_{s,\alpha,\bsgamma,q}^{\Sob})\right)\right] \\
    & \leq \frac{1}{b^m}\sum_{\emptyset\neq u\subseteq 1{:}s} \gamma^{\lambda}_u \left(\frac{b^2}{b-1}\right)^{|u|}C_{\alpha}^{\lambda |u|}A_{\alpha,\lambda}^{|u|}\prod_{j\in u} j\log_b(j+b) .
\end{align*}
Applying the inverse map $\phi^{-1}$ in \eqref{eq:epsilon}, the proof is completed.
\end{proof}

In the case of product weights, by using the elementary inequality $1+x\leq \exp(x)$, the probabilistic worst-case error bound shown in Theorem~\ref{thm:main2_net} is bounded by
\[ \epsilon \leq \inf_{\lambda\in (1/\alpha,1]}\left(\frac{1}{\delta b^m}\exp\left( \left(\frac{b^2}{b-1}\right)C_{\alpha}^{\lambda}A_{\alpha,\lambda}\sum_{j=1}^{s}\gamma^{\lambda}_j j\log_b(j+b)\right)\right)^{1/\lambda},\]
which leads to the following corollary.

\begin{corollary}\label{rem:main2_net}
    If there exists $\lambda'\in (1/\alpha,1]$ such that the product weights satisfy $\sum_{j=1}^{\infty}\gamma^{\lambda'}_j j\log_b(j+b)<\infty$, the probabilistic worst-case error given in Theorem~\ref{thm:main2_net} is bounded independently of $s$ and decays at the rate of $O(b^{-m/\lambda'})$.
\end{corollary}

Again a similar result holds for the set \eqref{eq:set2}. Although the proof is also similar to that of Theorem~\ref{thm:main2_net}, we give its sketch for the sake of completeness.

\begin{theorem}\label{thm:main2_plr}
For $\alpha\in \NN$, $\alpha\geq 2$, $1\leq q\leq \infty$ and a set of weights $\bsgamma$, let $B=\cF_{s,\alpha,\bsgamma,q}^{\Sob}$. For $m\in \NN$, let $S$ be a set of polynomial lattices \eqref{eq:set2}. In Proposition~\ref{prop:meta}, the assumption \eqref{eq:existence} holds for any $\delta\in (0,1)$ and 
\[ \epsilon = \inf_{\substack{\lambda\in (1/\alpha,1]\\ \tau\in (0,\lambda-1/\alpha)}}\left(\frac{3}{(b^m-1)\delta \tau e\log b}\sum_{\emptyset\neq u\subseteq 1{:}s} \gamma_u^{\lambda} C_{\alpha}^{\lambda|u|}A_{\alpha,\lambda-\tau}^{|u|}\right)^{1/\lambda},\]
where the constant $A_{\alpha,\lambda-\tau}$ is the same as in Theorem~\ref{thm:main2_net}.
\end{theorem}

\begin{proof}
Again let $\phi(x)=x^{\lambda}$ with $0<\lambda\leq 1$ in \eqref{eq:epsilon}. By applying Jensen's inequality \eqref{eq:jensen}, the linearity of expectation and Lemma~\ref{lem:plr_walsh_average}, we have
\begin{align*}
    \EE\left[ \phi\left(e^{\wor}(Q_{P(p,\bsg)}; \cF_{s,\alpha,\bsgamma,q}^{\Sob})\right)\right] & \leq \sum_{\emptyset\neq u\subseteq 1{:}s} \gamma^{\lambda}_u C_{\alpha}^{\lambda |u|}\sum_{\bsk_u\in \NN^{|u|}}b^{-\lambda \mu_{\alpha}(\bsk_u)}\prob{\bsk_u \in P_u^\perp(p,\bsg) }\\
    & \leq \frac{3}{b^m-1}\sum_{\emptyset\neq u\subseteq 1{:}s} \gamma_u^{\lambda} C_{\alpha}^{\lambda|u|}\sum_{\bsk_u\in \NN^{|u|}}b^{-\lambda\mu_{\alpha}(\bsk_u)}\mu_1(\bsk) \\
    & \leq \frac{3}{(b^m-1)\tau e\log b}\sum_{\emptyset\neq u\subseteq 1{:}s} \gamma_u^{\lambda} C_{\alpha}^{\lambda|u|}\sum_{\bsk_u\in \NN^{|u|}}b^{-(\lambda-\tau)\mu_{\alpha}(\bsk_u)}\\
    & = \frac{3}{(b^m-1)\tau e\log b}\sum_{\emptyset\neq u\subseteq 1{:}s} \gamma_u^{\lambda} C_{\alpha}^{\lambda|u|}A_{\alpha,\lambda-\tau}^{|u|},
\end{align*}
if $1/\alpha<\lambda\leq 1$ and $0<\tau<\lambda-1/\alpha$. Here, from the second to the third line, we used $\mu_1(\bsk)\leq \mu_{\alpha}(\bsk)$ (cf.\ Definitions~\ref{def:NRT_weight} and \ref{def:dick_weight}) and then applied the elementary inequality $x\leq b^{\tau x}/(\tau e\log b)$, which holds for any $\tau\in (0,1)$ and $x>0$, to the case $x=\mu_{\alpha}(\bsk)$. Applying the inverse map $\phi^{-1}$ in \eqref{eq:epsilon}, the proof is completed.
\end{proof}

We note that a quite similar result has been shown recently in \cite[Theorem~3.7]{GL22} for higher order polynomial lattices, i.e., polynomial lattices with a modulus of larger degree $n\geq m$. This subtle difference between point sets requires including the additional parameter $\tau$ in the present analysis, especially to deal with $\mu_1(\bsk)$ appearing in the first statement of Lemma~\ref{lem:plr_walsh_average}.

\begin{corollary}\label{rem:main2_plr}
If there exists $\lambda'\in (1/\alpha,1]$ such that the product weights satisfy $\sum_{j=1}^{\infty}\gamma^{\lambda'}_j <\infty$, the probabilistic worst-case error given in Theorem~\ref{thm:main2_plr} is bounded independently of $s$ and decays at the rate of $O(b^{-m/\lambda'})$.
\end{corollary}
\noindent Again, this sufficient condition compares favorably with that in Corollary~\ref{rem:main2_net}.

A lower error estimate by Sharygin \cite{Sh63}, originally for unweighted function spaces, applies to the weighed spaces $\cF^{\Sob}_{s,\alpha,\bsgamma,q}$, and the convergence rate of the worst-case error cannot be better than $N^{-\alpha}$. As the number of points is $N=b^m$ in Theorems~\ref{thm:main2_net} and \ref{thm:main2_plr}, with fixed values of $\delta$ and $b$, our obtained probabilistic worst-case error bounds are of order $N^{-1/\lambda}$ for $\lambda\in (1/\alpha,1]$, which is almost optimal as $\lambda\searrow 1/\alpha$ but with growing constants. Regarding the dimension-independence of the worst-case error, the sufficient condition same as that in Corollary~\ref{rem:main2_plr} can be derived from \cite[Theorem~3.10]{DKLNS14} for interlaced polynomial lattices. This condition was also shown to hold for extrapolated polynomial lattices in \cite[Corollary3.7]{DGY19}.

\subsubsection{Weighted space of infinitely smooth functions}

The last, third function space contains only infinitely many times differentiable functions, for which it can be expected from \cite{Su17,DGSY17,PO21,PO22} that the worst-case error decays faster than any polynomial convergence.
\begin{definition}[weighted space of infinitely smooth functions]\label{def:analytic_space}
For a sequence $\bsu=(u_1,u_2,\ldots)$ with $u_1\geq u_2\geq \cdots> 0$, the weighted space of infinitely smooth functions, denoted by $\cF_{s,\infty,\bsu}$, is a Banach space consisting of $f\in C^{\infty}([0,1]^s)$ with the norm
\[ \|f\|_{\cF_{s,\infty,\bsu}}:= \sup_{(\alpha_1,\ldots,\alpha_s)\in \NN_0^s}\prod_{j=1}^{s}u_j^{-\alpha_j}\int_{[0,1]^s}\left|f^{(\alpha_1,\ldots,\alpha_s)}(\bsx)\right|\rd \bsx.\]
\end{definition}

Similarly to the function space studied in section~\ref{subsubsec:main2}, the worst-case error in $\cF_{s,\infty,\bsu}$ has been analyzed for digital nets though Walsh analysis for infinitely smooth functions \cite{SY16,Yo17}. There the modified Dick weight with $\alpha=\infty$ plays a crucial role. 

\begin{definition}[modified infinite Dick weight]\label{lem:walsh_decay_infty}
Let $a$ be a real number. For $k\in \NN$, we denote its $b$-adic expansion by
\[ k=\kappa_1 b^{c_1-1}+\kappa_2 b^{c_2-1}+\cdots +\kappa_v b^{c_v-1},\]
with $\kappa_1,\ldots,\kappa_v\in \{1,\ldots,b-1\}$ and $c_1>c_2>\cdots>c_v>0$. The modified infinite Dick weight $\mu_{\infty,a}: \NN_0\to \RR$ is defined by $\mu_{\infty,a}(0)=0$ and
\[ \mu_{\infty,a}(k)=\sum_{i=1}^{v}\left(c_i+a\right)=\mu_{\infty}(k)+av.\]
Moreover, in the case of vectors in $\NN_0^s$, for a sequence $\bsa=(a_1,\ldots,a_s)\in \RR^s$, we define
\[ \mu_{\infty,\bsa}(\bsk) = \sum_{j=1}^{s}\mu_{\infty,a_j}(k_j).\]
\end{definition}

According to \cite[Lemma~6.5]{Su17}, the worst-case error for any digital net $P$ in $\cF_{s,\infty,\bsu}$ is bounded above by
\begin{align}\label{eq:wafom}
    e^{\wor}(Q_P; \cF_{s,\infty,\bsu})\leq \sum_{\bsk\in P^{\perp}\setminus \{\bszero\}}b^{-\mu_{\infty,\bsa}(\bsk)},
\end{align} 
where $\bsa=(a_1,a_2,\ldots)$ is a sequence defined by $a_j=-\log_b( K_b m_b^{-1}u_j),$ with $m_b=2\sin(\pi/b)$,
\[ M_b=\begin{cases} 2 & \text{if $b$ is even,} \\ 2\sin((b+1)\pi/2b) & \text{if $b$ is odd,} \end{cases}\]
and
\[ K_b=\begin{cases} 2 & \text{if $b=2$,} \\ M_b+bm_b/(b-M_b) & \text{if $b\geq 3$.} \end{cases}\]

First let $S$ be given by \eqref{eq:set1} with good square generating matrices. Unfortunately, our average argument in the sense of either \eqref{eq:markov} or \eqref{eq:epsilon} prevents us from capturing a precise convergence behavior (cf. \cite[Sections~6.2 \& 6.3]{Su17} and \cite[Sections~4 \& 5]{PO22}). The following result only shows that the probabilistic worst-case error decays faster than any polynomial convergence of a fixed degree.

\begin{theorem}\label{thm:main3_net}
For a sequence $\bsu=(u_1,u_2,\ldots)$ with $u_1\geq u_2\geq \cdots> 0$, let $B=\cF_{s,\infty,\bsu}$. For $m\in \NN$, let $S$ be a set of linearly scrambled digital nets \eqref{eq:set1} with $C_1,\ldots,C_s\in \FF_b^{m\times m}$ all non-singular and satisfying \eqref{eq:tvalue_Niederreiter}. In Proposition~\ref{prop:meta}, the assumption \eqref{eq:existence} holds for any $\delta\in (0,1)$ and 
\[ \epsilon = \inf_{\lambda\in (0,1]}\left(\frac{1}{\delta b^m}\left[ -1+\prod_{j=1}^{s}\left(1+\left(\frac{b}{b-1}\right)A_{\infty,\lambda}b^{-\lambda a_j}j \log_b(j+b)\right)\right]\right)^{1/\lambda},\]
where $a_j=-\log_b(K_b m_b^{-1}u_j),$ and $A_{\infty,\lambda}$ is finite for any $\lambda\in (0,1)$ and equals
\[ A_{\infty,\lambda} :=\sum_{k=1}^{\infty}b^{-\lambda\mu_\infty(k)} = \prod_{\ell=1}^{\infty}\left[ 1+\frac{b-1}{b^{\lambda \ell}}\right]-1.\]
\end{theorem}
\noindent Note that the constant $A_{\infty,\lambda}$ equals $\lim_{\alpha\to \infty}A_{\alpha,\lambda}$, where $A_{\alpha,\lambda}$ is as in Theorem~\ref{thm:main2_net}.

\begin{proof}
Let us consider \eqref{eq:epsilon} with the choice $\phi(x)=x^{\lambda}$ for $0<\lambda\leq 1$. Following an argument similar to the proof of Theorem~\ref{thm:main2_net}, we have
\begin{align*}
    & \EE\left[ \phi\left(e^{\wor}(Q_{P(L_1C_1,\ldots,L_sC_s)}; \cF_{s,\infty,\bsu})\right)\right] \\
    & \leq \sum_{\bsk\in \NN_0^s\setminus \{\bszero\}} \prob{\bsk \in P^\perp(L_1C_1,\ldots,L_sC_s) }b^{-\lambda \mu_{\infty,\bsa}(\bsk)}\\
    & \leq \frac{1}{b^m}\sum_{\emptyset \neq u\subseteq 1{:}s}\left(\frac{b}{b-1}\right)^{|u|}b^{t_u}\sum_{\bsk_u\in \NN^{|u|}}b^{-\lambda \mu_{\infty,\bsa}(\bsk_u)}\\
    & \leq \frac{1}{b^m}\sum_{\emptyset \neq u\subseteq 1{:}s}\left(\frac{b}{b-1}\right)^{|u|}b^{t_u}b^{-\lambda|\bsa_u|_1}\sum_{\bsk_u\in \NN^{|u|}}b^{-\lambda \mu_{\infty}(\bsk_u)}.
\end{align*}
Here, in a way similar to \cite[Proposition~3]{DGSY17}, we can prove that the inner sum over $\bsk_u$ equals $A_{\infty,\lambda}^{|u|}$. Finally, by using a bound \eqref{eq:tvalue_Niederreiter} on the $t$-value of the projected point set, we have
\begin{align*}
    & \EE\left[ \phi\left(e^{\wor}(Q_{P(L_1C_1,\ldots,L_sC_s)}; \cF_{s,\infty,\bsu})\right)\right] \\
    & \leq \frac{1}{b^m}\sum_{\emptyset \neq u\subseteq 1{:}s}\left(\frac{b}{b-1}\right)^{|u|}A_{\infty,\lambda}^{|u|}\prod_{j\in u}b^{-\lambda a_j}j \log_b(j+b)\\
    & = \frac{1}{b^m}\left[ -1+\prod_{j=1}^{s}\left(1+\left(\frac{b}{b-1}\right)A_{\infty,\lambda}b^{-\lambda a_j}j \log_b(j+b)\right)\right].
\end{align*}
Applying the inverse map $\phi^{-1}$ in \eqref{eq:epsilon} completes the proof.
\end{proof}

Recalling that $a_j=-\log_b( K_b m_b^{-1}u_j),$ the following result holds:
\begin{corollary}
If the weight sequence $\bsu$ satisfies $\sum_{j=1}^{\infty}u_j^{\lambda}\, j\log_b(j+b)<\infty$ for any $\lambda\in (0,1],$ the probabilistic worst-case error $\epsilon$ given in Theorem~\ref{thm:main3_net} is bounded independently of the dimension $s$ and decays faster than any polynomial convergence of a fixed degree.
\end{corollary}

For the set \eqref{eq:set2} it is possible to evaluate a convergence behavior more precisely under some conditions on the weights $\bsu$. To do so, as in \cite{DGSY17}, let us consider a family of concave functions of the form $\phi(x)=b^{-(-\log_b x)^{\lambda}}$ with $0<\lambda\leq 1$, which maps $b^{-x}$ to $b^{-x^{\lambda}}$. In order to make $\phi$ concave and unbounded monotonically increasing over the support $[0,\infty)$, let 
\[ \phi(x)=\begin{cases}
2^{-(-\log_2 x)^{\lambda}} & \text{if $0<x\leq \tilde{x}_{\lambda}$,} \\ \frac{\lambda (\log_2 \tilde{x}_{\lambda})^{\lambda-1}}{e\tilde{x}_{\lambda}}(x-\tilde{x}_{\lambda})+\frac{1}{e} & \text{otherwise,}
\end{cases}\]
with $\tilde{x}_{\lambda}=2^{-(\log 2)^{1/\lambda}}$ for $b=2$, and
\[ \phi(x)=\begin{cases}
b^{-(-\log_b x)^{\lambda}} & \text{if $0<x\leq 1/b$,} \\ \lambda(x-\frac{1}{b})+\frac{1}{b} & \text{otherwise,}
\end{cases}\]
for $b\geq 3$. We set $\phi(0)=0$ for any base $b\geq 2$.\footnote{In \cite{DGSY17}, there are some typos in the definition of $\phi$, which we correct in this paper.} In what follows, we focus on the case $b\geq 3$ and assume $a_1\geq 0$, i.e., $u_1\leq m_b/K_b$. Note that a similar result without the assumption $a_1\geq 0$ can be proven in the same way, including the case $b=2$ as in \cite{DGSY17,Su17}. Then, for any $\bsk\in \NN_0^s\setminus \{\bszero\}$, we have $\mu_{\infty,\bsa}(\bsk)\geq \mu_1(\bsk)\geq 1$ and $b^{-\mu_{\infty,\bsa}(\bsk)}\leq 1/b$. Now we show the following result.

\begin{theorem}\label{thm:main3_plr}
For a sequence $\bsu=(u_1,u_2,\ldots)$ with $m_b/K_b\geq u_1\geq u_2\geq \cdots> 0$, let $B=\cF_{s,\infty,\bsu}$ and $a_j=-\log_b(K_b m_b^{-1}u_j).$ For $m\in \NN$, let $S$ be a set of polynomial lattices \eqref{eq:set2} with $b\geq 3$ and let $\phi$ be given as above. In Proposition~\ref{prop:meta}, the assumption \eqref{eq:existence} holds for any $\delta\in (0,1)$ and the following value of $\epsilon$.

\begin{enumerate}
    \item (unweighted case) If there exists a constant $a\geq 0$ such that $a_1=a_2=\cdots=a$, 
    \[ \epsilon = \inf_{\substack{\lambda\in (1/2,1)\\ \tau\in (0,\min(\lambda,1/(\log b)))}}\phi^{-1}\left( \frac{3C_{s,\lambda,\tau}}{(b^m-1)\delta(\tau e\log b)^{1/\tau}}\right),\]
    where
    \[ C_{s,\lambda,\tau}=\sum_{i=1}^{\infty}\exp\left( 2\sqrt{s(b-1)(i+1)}-(i^\lambda-i^\tau)\log b \right).\]
    \item (weighted case) If there exist constants $a,q> 0$ such that $a_j\geq a(j-1)^q$, 
    \[ \epsilon = \inf_{\substack{\lambda\in ((q+1)/(2q+1),1)\\ \tau\in (0,\min(\lambda,1/(\log b)))}}\phi^{-1}\left( \frac{3C_{a,q,\lambda,\tau}}{(b^m-1)\delta(\tau e\log b)^{1/\tau}}\right),\]
    where 
    \[ C_{a,q,\lambda,\tau}=\sum_{i=1}^{\infty}\exp\left( A_{a,q} (i+1)^{(q+1)/(2q+1)}-(i^\lambda-i^\tau)\log b \right)\]
    with $A_{a,q}=1+(b-1)\left(1+\frac{\Gamma(1/q)}{qa^{1/q}}\right)$ and $\Gamma$ being the Gamma function.
\end{enumerate}
\end{theorem}

\begin{proof}
Similarly to the proof of Theorem~\ref{thm:main2_plr}, by using the elementary inequality $x\leq b^{x^{\tau}}/(\tau e\log b)^{1/\tau}$ which holds for any $\tau\in (0,1/(\log b))$ and $x\geq 1$ instead of $x\leq b^{\tau x}/(\tau e\log b)$, we have
\begin{align}
    \EE\left[ \phi\left(e^{\wor}(Q_{P(p,\bsg)}; \cF_{s,\infty,\bsu})\right)\right] & \leq \sum_{\bsk\in \NN_0^s\setminus \{\bszero\}}\phi\left(b^{-\mu_{\infty,\bsa}(\bsk)}\right)\prob{\bsk\in P^{\perp}(p,\bsg)} \notag \\
    & \leq \frac{3}{b^m-1}\sum_{\bsk\in \NN_0^s\setminus \{\bszero\}}b^{-(\mu_{\infty,\bsa}(\bsk))^\lambda}\mu_1(\bsk)\notag \\
    & \leq \frac{3}{(b^m-1)(\tau e\log b)^{1/\tau}}\sum_{\bsk\in \NN_0^s\setminus \{\bszero\}}b^{-(\mu_{\infty,\bsa}(\bsk))^\lambda+(\mu_{\infty,\bsa}(\bsk))^\tau}. \label{eq:bound_infinite_smooth}
\end{align}

For $i\in \NN$, let us define
\[ \vol_{\bsa}(i):=\left| \left\{ \bsk\in \NN_0^s\setminus \{\bszero\}\mid i\leq \mu_{\infty,\bsa}(\bsk)< i+1\right\}\right|.\]
Then, for any $x\in (0,1)$, it holds that
\begin{align*}
    \vol_{\bsa}(i) & = \sum_{\substack{\bsk\in \NN_0^s\setminus \{\bszero\}\\ i\leq \mu_{\infty,\bsa}(\bsk)< i+1}}1\leq \sum_{\substack{\bsk\in \NN_0^s\setminus \{\bszero\}\\ \mu_{\infty,\bsa}(\bsk)< i+1}}1 \leq \sum_{\substack{\bsk\in \NN_0^s\setminus \{\bszero\}\\ \mu_{\infty,\bsa}(\bsk)< i+1}}x^{\mu_{\infty,\bsa}(\bsk)-(i+1)}\\
    & \leq \sum_{\bsk\in \NN_0^s}x^{\mu_{\infty,\bsa}(\bsk)-(i+1)} = \frac{1}{x^{i+1}}\prod_{j=1}^{s}\prod_{\ell=1}^{\infty}\left( 1+(b-1)x^{\ell+a_j}\right),
\end{align*}
where the last equality follows from \cite[Eq.~(4)]{Su17}. Taking the natural logarithm of the left-most and right-most sides and using the elementary inequality $\log(1+x)\leq x$ for $x>0$, we have
\begin{align*}
    \log \vol_{\bsa}(i) & \leq -(i+1)\log x+\sum_{j=1}^{s}\sum_{\ell=1}^{\infty}\log \left( 1+(b-1)x^{\ell+a_j}\right)\\
    & \leq -(i+1)\log x+(b-1)\sum_{j=1}^{s}x^{a_j}\sum_{\ell=1}^{\infty}x^{\ell}\\
    & = -(i+1)\log x+\frac{(b-1)x}{1-x}\sum_{j=1}^{s}x^{a_j} \leq -(i+1)\log x-\frac{b-1}{\log x}\sum_{j=1}^{s}x^{a_j}.
\end{align*}
In what follows, we discuss two different cases separately.
\begin{enumerate}
    \item (unweighted case) If $a_1=a_2=\cdots =a\geq 0$, by choosing 
    \[ x=1/\exp(\sqrt{s(b-1)/(i+1)}),\]
    we obtain
    \[ \vol_{\bsa}(i)\leq \exp\left( 2\sqrt{s(b-1)(i+1)}\right).\]
    For any $0<\tau<\min(\lambda,1/(\log b))$, the sum over $\bsk$ of \eqref{eq:bound_infinite_smooth} is bounded by
    \begin{align*}
    & \sum_{\bsk\in \NN_0^s\setminus \{\bszero\}}b^{-(\mu_{\infty,\bsa}(\bsk))^\lambda+(\mu_{\infty,\bsa}(\bsk))^\tau} \\
    & = \sum_{i=1}^{\infty}\sum_{\substack{\bsk\in \NN_0^s\setminus \{\bszero\}\\ i\leq \mu_{\infty,\bsa}(\bsk)< i+1}}b^{-(\mu_{\infty,\bsa}(\bsk))^\lambda+(\mu_{\infty,\bsa}(\bsk))^\tau} \leq \sum_{i=1}^{\infty}b^{-i^\lambda+i^\tau}\vol_{\bsa}(i) \\
    & \leq \sum_{i=1}^{\infty}\exp\left( 2\sqrt{s(b-1)(i+1)}-(i^\lambda-i^\tau)\log b \right) = C_{s,\lambda,\tau}.
    \end{align*}
    This infinite sum converges when $\lambda>1/2$ and $0<\tau<\min(\lambda,1/(\log b))$.
    \item (weighted case) If there exist $a,q>0$ such that $a_j\geq a(j-1)^q$ for all $j$, the sum of $x^{a_j}$ is bounded by 
    \begin{align*}
    \sum_{j=1}^{s}x^{a_j} & \leq \sum_{j=1}^{s}x^{a(j-1)^q}\leq 1+\int_0^{s-1}x^{at^q}\rd t\leq 1+\int_0^{\infty}x^{at^q}\rd t\\
    & = 1+\frac{\Gamma(1/q)}{qa^{1/q}(-\log x)^{1/q}},
    \end{align*}
    where we used the result from \cite[Lemma~6.11]{Su17} for the last equality. By choosing $x=1/\exp((i+1)^{-q/(2q+1)})$, we obtain
    \[ \vol_{\bsa}(i)\leq \exp\left(A_{a,q} (i+1)^{(q+1)/(2q+1)}\right),\]
    with 
    \[ A_{a,q}=1+(b-1)\left(1+\frac{\Gamma(1/q)}{qa^{1/q}}\right).\]
    For any $0<\tau<\min(\lambda,1/(\log b))$, the sum over $\bsk$ of \eqref{eq:bound_infinite_smooth} is bounded by
    \begin{align*}
    & \sum_{\bsk\in \NN_0^s\setminus \{\bszero\}}b^{-(\mu_{\infty,\bsa}(\bsk))^\lambda+(\mu_{\infty,\bsa}(\bsk))^\tau} \\
    & \leq \sum_{i=1}^{\infty}\exp\left( A_{a,q} (i+1)^{(q+1)/(2q+1)}-(i^\lambda-i^\tau)\log b \right)=C_{a,q,\lambda,\tau},
    \end{align*}
    which converges when $\lambda>(q+1)/(2q+1)$ and $0<\tau<\min(\lambda,1/(\log b))$.
\end{enumerate}
\noindent 
The result of the theorem follows from \eqref{eq:epsilon}.
\end{proof}

For the unweighted case, the constant $C_{s,\lambda,\tau}$ depends exponentially on $s$. For $m$ being large enough such that there exist $\lambda', \tau'$ that satisfy
\[ \frac{3C_{s,\lambda',\tau'}}{(b^m-1)\delta(\tau' e\log b)^{1/\tau'}}\leq \frac{1}{b},\]
we have
\[ \epsilon=b^{-\left(\log_b\left((b^m-1)\delta(\tau' e\log b)^{1/\tau'} \big/ (3C_{s,\lambda',\tau'})\right)\right)^{1/\lambda'}}. \]
This implies that the probabilistic worst-case error asymptotically decays faster than any polynomial convergence. However, due to the exponential dependence of the constant $C_{s,\lambda',\delta'}$ on $s$, such an asymptotic regime can be only achieved for exponentially large $b^m$ in terms of $s$. On the contrary, the following holds for the weighted case:
\begin{corollary}
If there exist $a,q>0$ such that $u_j\leq (m_b/K_b)\exp(-a(j-1)^q)$ for all $j$, the probabilistic worst-case error $\epsilon$ is independent of the dimension $s$, as given in the second item of Theorem~\ref{thm:main3_plr}, and  decays super-polynomially as
\[ \epsilon=b^{-\left(\log_b\left((b^m-1)\delta(\tau e\log b)^{1/\tau} \big/ (3C_{a,q,\lambda,\tau})\right)\right)^{1/\lambda}}, \]
for any $(q+1)/(2q+1)<\lambda\leq 1$ and $0<\tau<\min(\lambda,1/(\log b))$.
\end{corollary}

In the literature, a computable inversion formula for the worst-case error bound \eqref{eq:wafom} has been studied in the unweighted cases in \cite{MSM14,Su15}. This formula was used in \cite{Ha15} to optimize the linear scrambling of digital nets, and Theorem~\ref{thm:main3_net} provides partial theoretical support for such optimization. In the weighted cases, a dimension-independent super-polynomial convergence has been established under a slightly milder condition $\liminf_{j\to \infty}\log(u_j^{-1})/j^q>0$ for some $q>0$ in \cite{Su17} by a non-constructive existence proof and in \cite{DGSY17} for interlaced polynomial lattices with a carefully chosen interlacing factor. It should be noted that a lower bound on the worst-case error in $\cF_{s,\infty,\bsu}$ is currently unknown. Instead, a lower bound in the Walsh space, into which $\cF_{s,\infty,\bsu}$ is embedded, is shown in \cite[Section~5]{Su17}, and the super-polynomial convergence rates obtained in Theorem~\ref{thm:main3_plr} closely match that result.
\section{Numerical experiments}\label{sec:numerics}

We conclude this paper with numerical experiments. Although our theoretical results in section~\ref{sec:theory} are obtained for the infinite-precision setting $w=\infty$, we use finite-precision versions with $b=2$:
\[ \left\{ P(L_1C_1,\ldots,L_sC_s)\;\middle|\; L_1,\ldots,L_s\in \cL_{w=52,m}\right\}, \]
with the generating matrices of Sobol' points $C_1,\ldots,C_s\in \FF_2^{m\times m}$, and
\[ \left\{ P(p,\bsg,w=52)\;\middle|\; p\in \PP_m, \bsg\in \GG_m^s\right\}, \]
instead of \eqref{eq:set1} and \eqref{eq:set2} for the set $S$, respectively, throughout this section. The number $w=52$ not only seems large enough to approximate those results well but also enables us to implement reasonably in the double-precision floating-point format. Moreover, we fix the number of random draws at $r=15$, whose choice is justified by a rule of thumb given in the following remark. 

\begin{remark}
To reach a certain confidence level $1-\theta$ with small $\theta<1/4$, we use Proposition~\ref{prop:meta} to obtain $2^{r-1}\delta^{(r+1)/2}\leq \theta$. Let us set $\delta=(4\theta)^{2/r}/4$ to satisfy this inequality. Now, given a total budget of $N=r b^m$ function evaluations, consider maximizing the product $\delta b^m=N(4\theta)^{2/r}/(4r)$ with respect to $r$. Since the product $\delta b^m$ commonly appears in the denominator of $\epsilon$ in the theorems presented in the previous section, it is a natural quantity to optimize. This maximization leads to $r=-2\log(4\theta),$ which corresponds to $\delta=1/(4e).$ Using this rule of thumb, we obtain $\theta=1/(4e^{r/2})$ for a given $r$. Therefore, choosing $r=15$ ensures that the worst-case error is stabilized below a guaranteed probability of failure $\theta<1/7000.$ By changing the value of $r$, one can adjust to different confidence levels.
\end{remark}

\noindent The numerical results for a series of four different experiments are reported below.
\begin{example}
We begin by examining the distribution of worst-case errors for randomly chosen point sets, here focusing only on the set of polynomial lattices in the weighted Sobolev space of high order $\cF^{\Sob}_{s,\alpha,\bsgamma,q}$. Note that the worst-case error is bounded as shown in \eqref{eq:error_bound_high_order}. For the case of product weights, it follows from \cite[Lemma~1 \& Corollary~1]{BDLNP12} that this bound can be rewritten as
\[ \Scal_{\alpha,\bsgamma}(P)=-1+\frac{1}{|P|}\sum_{\bsx\in P}\prod_{j=1}^{s}\left[ 1+\gamma_j C_{\alpha}\omega_{\alpha}(x_{j})\right], \]
with $\omega_{\alpha}: [0,1)\to \RR$ having a concise expression if $b=2$ and $\alpha=2$ or $3$. We consider the cases $\alpha=2$ and $\alpha=3$, and set $m=8$, $s=5$, and $\gamma_j=C_{\alpha}^{-1}j^{-\alpha-1}$ for $j=1,\ldots,s$. For each value of $\alpha$, we generate $10^4$ polynomial lattices independently and randomly from the set and compute $\Scal_{\alpha,\bsgamma}$ for each polynomial lattice. The left panels of Figure~\ref{fig:worst-case} display histograms of the resulting values of $\log_2 \Scal_{\alpha,\bsgamma}$. For both values of $\alpha$, the obtained distributions exhibit a right-skewed shape, indicating that while the worst-case error is small for the majority of polynomial lattices, there is a small fraction of ``bad'' polynomial lattices with significantly larger worst-case errors.

We now evaluate how taking the median can help stabilize the worst-case error experimentally. From the proof of Proposition~\ref{prop:meta}, we can see that the worst-case error for the median-of-means with point sets $P_1, \dots, P_r$ is bounded above by the median of the worst-case errors for each individual point set. Justified by this, with $r=15$, we generate $10^4$ random sets of $r$ polynomial lattices and compute the median of the worst-case errors associated with each set. The results are presented in the right panels of Figure~\ref{fig:worst-case}, where we plot the histograms of all the realizations of 
\[ \log_2 \median\left(\Scal_{\alpha,\bsgamma}(P_1),\dots,\Scal_{\alpha,\bsgamma}(P_r)\right)=\median\left(\log_2 \Scal_{\alpha,\bsgamma}(P_1),\dots,\log_2 \Scal_{\alpha,\bsgamma}(P_r)\right). \]
We observe that the resulting distributions are more concentrated around small error values, indicating that taking the median successfully mitigates the influence of the ``bad'' polynomial lattices, exhibiting the effectiveness of the median-of-means approach in stabilizing the worst-case error.

\begin{figure}
    \centering
    \subfloat[$\alpha=2, r=1$]{\includegraphics[width=0.45\textwidth]{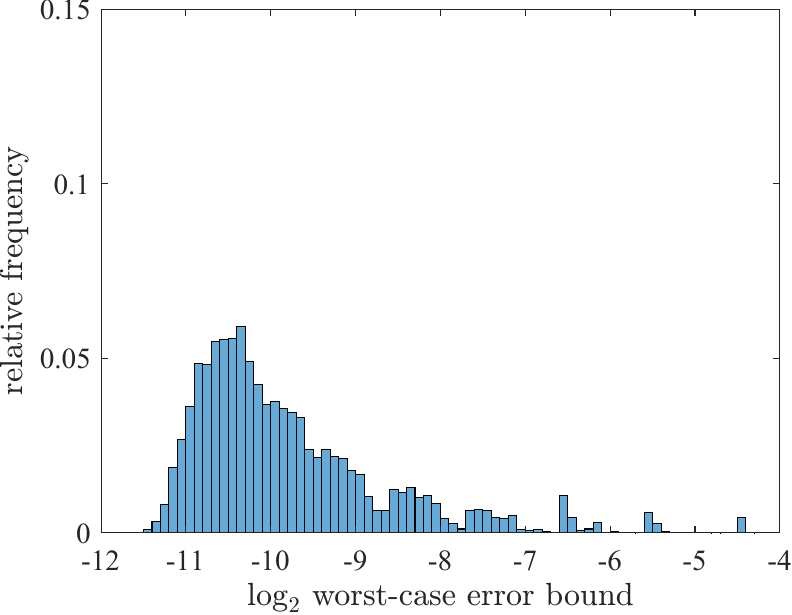}}
    \subfloat[$\alpha=2, r=15$]{\includegraphics[width=0.45\textwidth]{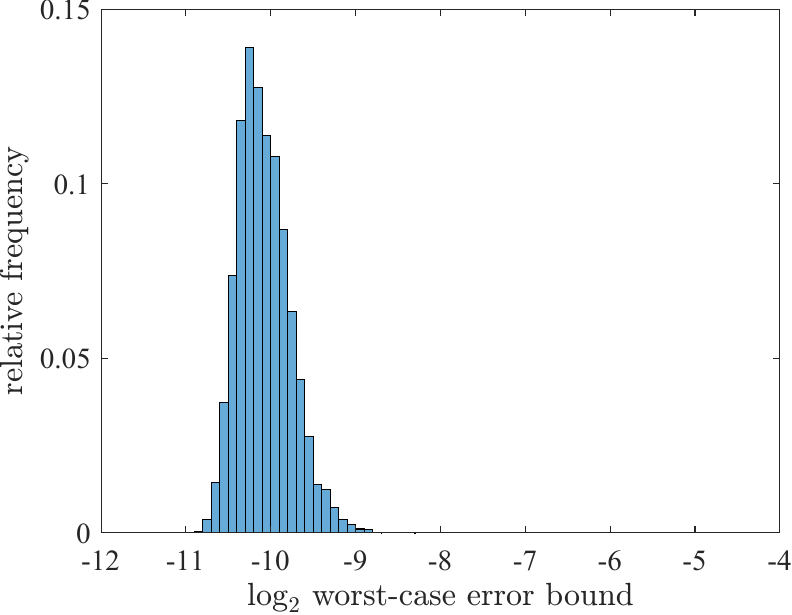}}\\
    \subfloat[$\alpha=3, r=1$]{\includegraphics[width=0.45\textwidth]{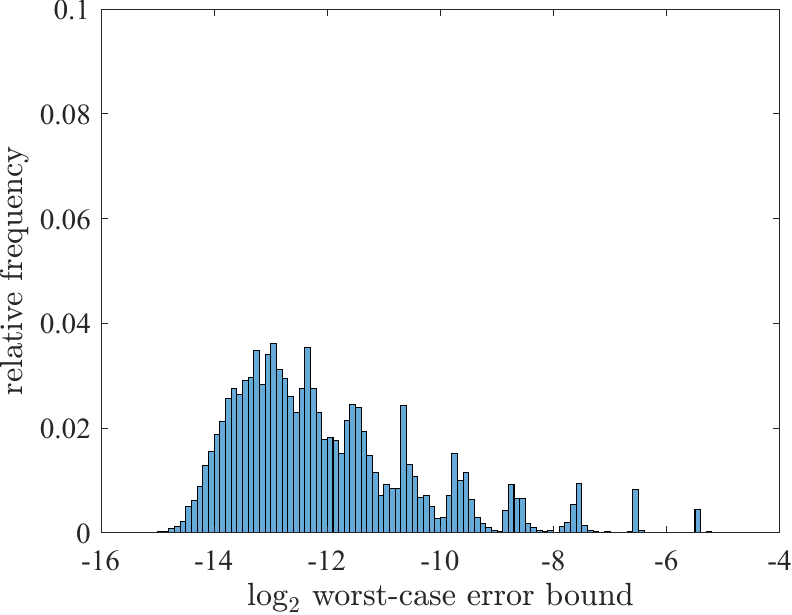}}
    \subfloat[$\alpha=3, r=15$]{\includegraphics[width=0.45\textwidth]{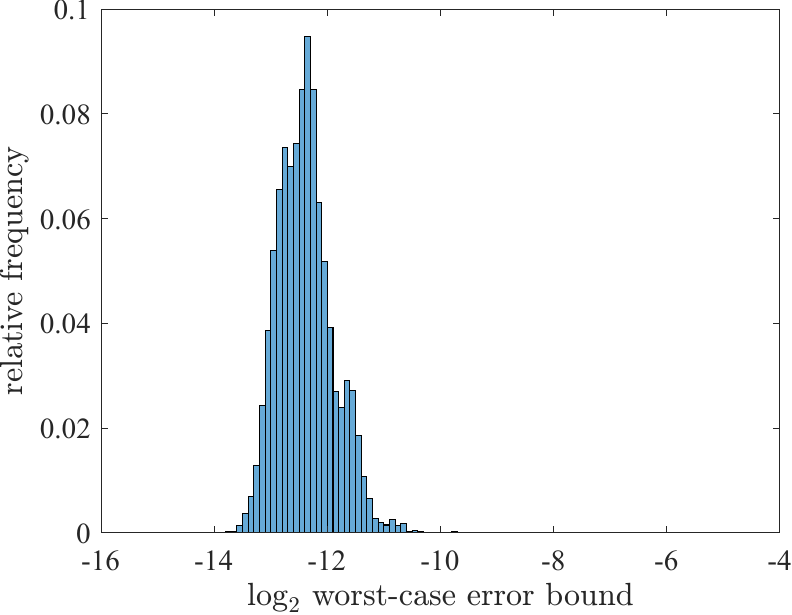}}
    \caption{Histograms of $\log_2 \Scal_{\alpha,\bsgamma}$ for randomly generated polynomial lattices (left panels) and those of the median of $\log_2 \Scal_{\alpha,\bsgamma}$ for randomly generated sets of $r=15$ polynomial lattices (right panels). The upper panels correspond to the case $\alpha=2$, while the lower panels correspond to $\alpha=3$.}
    \label{fig:worst-case}
\end{figure}
\end{example}

\begin{example}
The rest of the examples is devoted to illustrating how our median QMC integration works for some concrete test functions. We compare our novel two median QMC rules with the deterministic QMC rule using Sobol' point sets as one of the ``gold standard" methods. For this second experiment, we look at the simplest, one-dimensional integration problem. As we showed a universality for the three different function spaces in section~\ref{sec:theory}, we prepare one test function from the respective function space:
\begin{align*}
    f_1(x) & = \sqrt{x},\\
    f_2(x) & = x^2\left(\log x+\frac{1}{3}\right),\\
    f_3(x) & = x\exp(x).
\end{align*}
For any $\gamma_1>0$, it is easy to check that $f_1\in \cF_{1,1,\gamma_1}^{\Sob}$ but $f_1\notin \cF_{1,2,\gamma_1,q}^{\Sob}$ for any $1\leq q\leq \infty$. Similarly, we can check that $f_2\in \cF_{1,2,\gamma_1,q}^{\Sob}$ for any finite $q\geq 1$ but $f_2\notin \cF_{1,3,\gamma_1,1}^{\Sob}$, and also that $f_3\in \cF_{1,\infty,u_1}$ for any $u_1>0$. Note that, for the case $s=1$, the first $2^m$ points of a Sobol' sequence is just the equi-spaced points $\{i/2^m\mid 0\leq i<2^m\}$. 

The results for the three one-dimensional test functions are shown in Figure~\ref{fig:test_1d}. As can be seen from the left panel, Sobol' points cannot exploit the smoothness of integrands so that the error decays at the rate of $N^{-1}$ for all the test functions. On the contrary, our median QMC rules universally exploit the smoothness of integrands and exhibit a significantly better convergence behavior. In particular, for the function $f_3$, the rate of convergence looks accelerating until the error drops down to $2^{-52}$, which agrees with a super-polynomial convergence proven in Theorem~\ref{thm:main3_plr}. It is interesting to see that the two median QMC rules perform almost equally for all the test functions.

\begin{figure}
    \centering
    \subfloat[Unscrambled Sobol']{\includegraphics[width=0.32\textwidth]{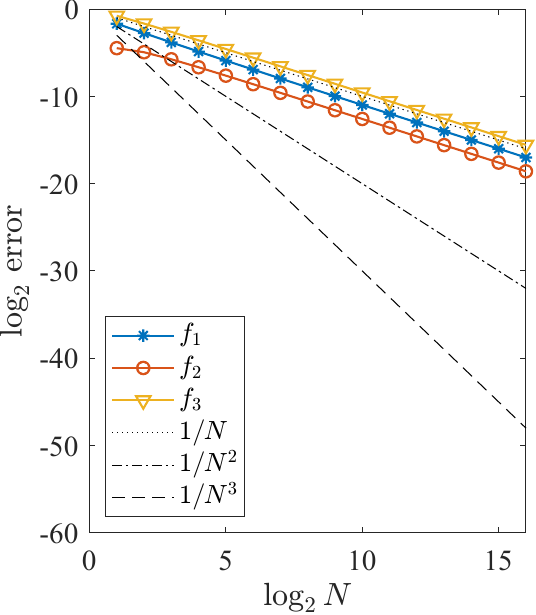}}
    \subfloat[Median scrambled Sobol']{\includegraphics[width=0.32\textwidth]{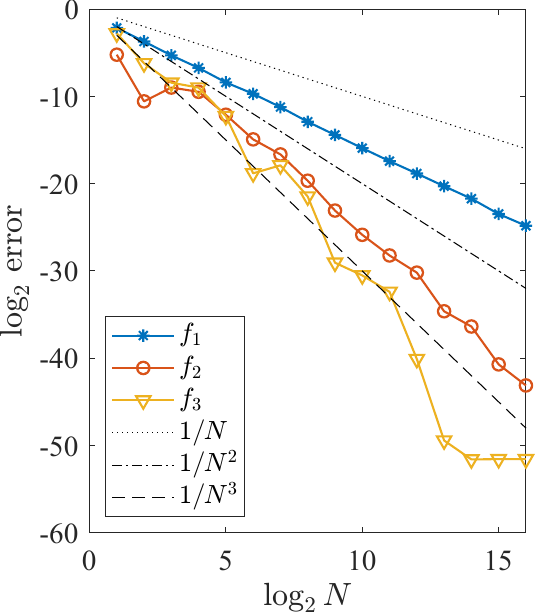}}
    \subfloat[Median polynomial lattices]{\includegraphics[width=0.32\textwidth]{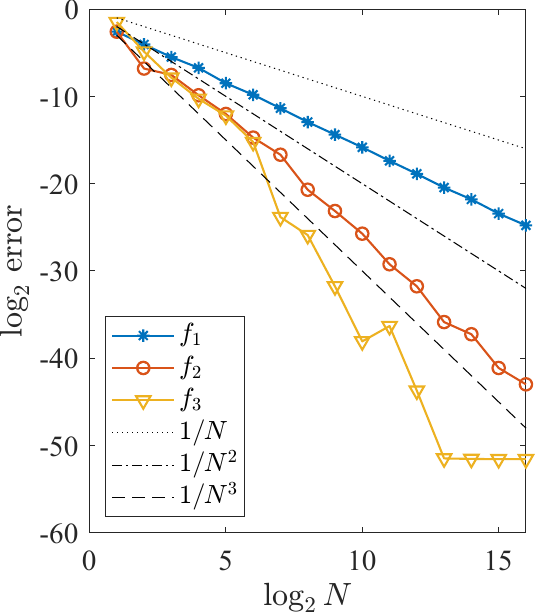}}
    \caption{Comparison of the one-dimensional integration error by QMC rule using Sobol' points (left) and two median QMC rules using linearly scrambled Sobol' points (middle) and randomly chosen polynomial lattice point sets (right) for the test functions $f_1$ (blue), $f_2$ (orange) and $f_3$ (yellow), respectively.}
    \label{fig:test_1d}
\end{figure}
\end{example}

\begin{example}
Let us move on to high-dimensional integration problems. Let $s=20$ and consider the following test function
\[ f_4(\bsx)=\prod_{j=1}^{s}\left[ 1+\frac{1}{\exp(\lceil c\rceil j)}\left(x_j^{c}-\frac{1}{1+c}\right)\right],\]
with a parameter $c>0$. When $c$ is not an integer, the $\lceil c\rceil$-th derivative of the function $x\mapsto x^{c}$ is not absolutely continuous but is in $L_q([0,1])$ for any $1\leq q<1/(\lceil c\rceil-c)$, which means that $f_4\in \cF_{s,1,\bsgamma}^{\Sob}$ if $0<c<1$ and $f_4\in \cF_{s,\lceil c\rceil,\bsgamma,q}^{\Sob}$ for any $1\leq q<1/(\lceil c\rceil-c)$ if $c>1$ is not an integer. This way, $\lceil c\rceil$ corresponds to the smoothness of this function. The factor $1/\exp(\lceil c\rceil j))$ is multiplied to model different relative importance for the individual variables.

The results for $f_4$ with $c=0.5, 1.5, 2.5$ are shown in Figure~\ref{fig:test_20d_finite}. Here again, the left panel clearly depicts that Sobol' points cannot exploit the smoothness of integrands and the error decays at the rate of $N^{-1}$ irrespective of the value of $c$. Both of the two median QMC rules compare quite favorably even for such a high-dimensional setting and both achieve an error decay nearly of order $N^{-\lceil c\rceil}$. Although taking the median has a stabilizing effect, a slight fluctuation in convergence is observed due to the fact that we randomly drew $r=15$ point sets only once for each value of $m=\log_2N$. Here again, the difference between the two median QMC rules is not significant.

\begin{figure}
    \centering
    \subfloat[Unscrambled Sobol']{\includegraphics[width=0.32\textwidth]{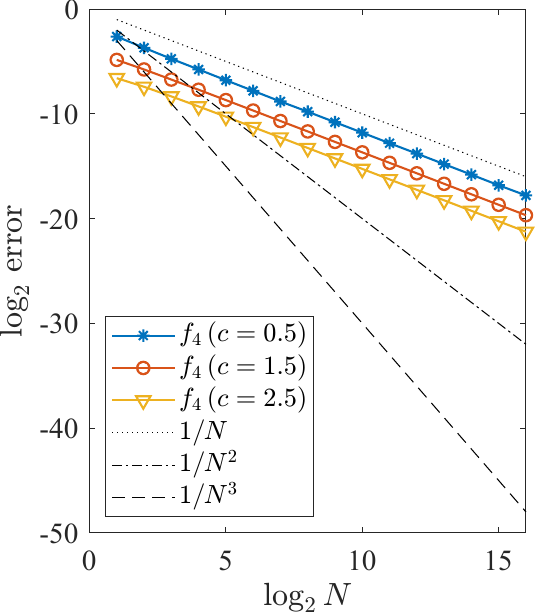}}
    \subfloat[Median scrambled Sobol']{\includegraphics[width=0.32\textwidth]{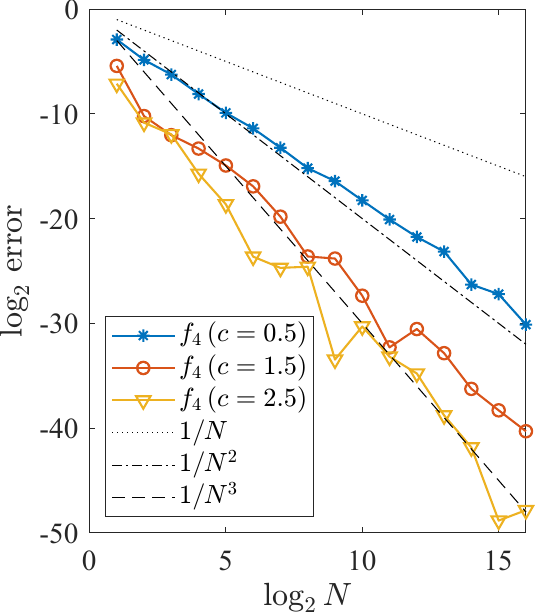}}
    \subfloat[Median polynomial lattices]{\includegraphics[width=0.32\textwidth]{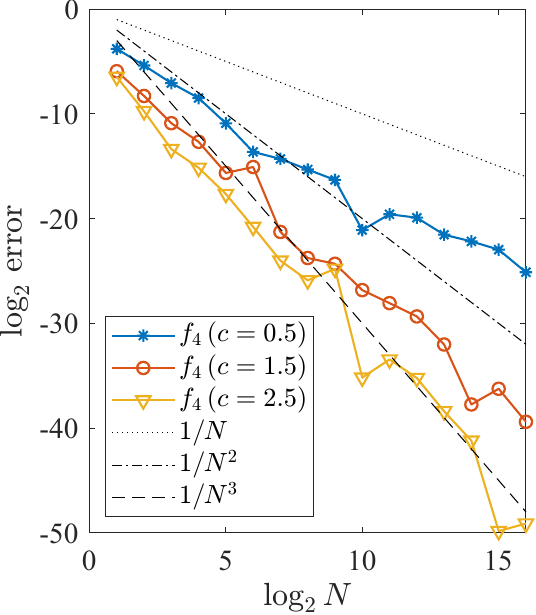}}
    \caption{Comparison of the 20-dimensional integration error by QMC rule using Sobol' points (left) and two median QMC rules using linearly scrambled Sobol' points (middle) and randomly chosen polynomial lattice point sets (right) for the finitely smooth function $f_4$ with $c=0.5$ (blue), $c=1.5$ (orange) and $c=2.5$ (yellow), respectively.}
    \label{fig:test_20d_finite}
\end{figure}
\end{example}

\begin{example}
Finally, let us consider the multivariate test function with infinite smoothness
\[ f_5(\bsx)=\exp\left( -\sum_{j=1}^{s}\frac{x_j}{2^{j^c}}\right),\]
with a parameter $c\geq 0$. It is easy to see that $f_5\in \cF_{s,\infty,\bsu}$ with $u_j=1/2^{j^c}$. Since we have $a_j=-\log_2(u_j)=j^c$, the case $c=0$ is classified into the unweighted case with $a=1$ in Theorem~\ref{thm:main3_plr}, while the case $c>0$ is into the weighted case with $a=1$ and $q=c$. In passing, it was shown in \cite{MOY18} that the worst-case error for a digital net over $\FF_2$ in $\cF_{s,\infty,\bsu}$ is bounded from above and below (up to constants) by the absolute integration error for the integrand $f_5$ with $1/2^{j^c}$ replaced by $u_j$. Therefore, the experimental results presented below on this function $f_5$ are expected to correspond to the behavior of the worst-case error.

The results for $f_5$ with $s=5$ and $c=0, 1, 2$ are shown in Figure~\ref{fig:test_5d_infinite}. Regarding each of the two median QMC rules, although it is not obvious from the figure whether the error decays super-polynomially within the performed range of $N$, the error decays at a rate between $N^{-1}$ and $N^{-2}$ even for the worst case $c=0$, which is better than that for naive QMC rule using Sobol' points. As $c$ increases, the convergence behavior for small $N$ gets improved and we can observe a higher order convergence than the case $c=0$.

\begin{figure}
    \centering
    \subfloat[Unscrambled Sobol']{\includegraphics[width=0.32\textwidth]{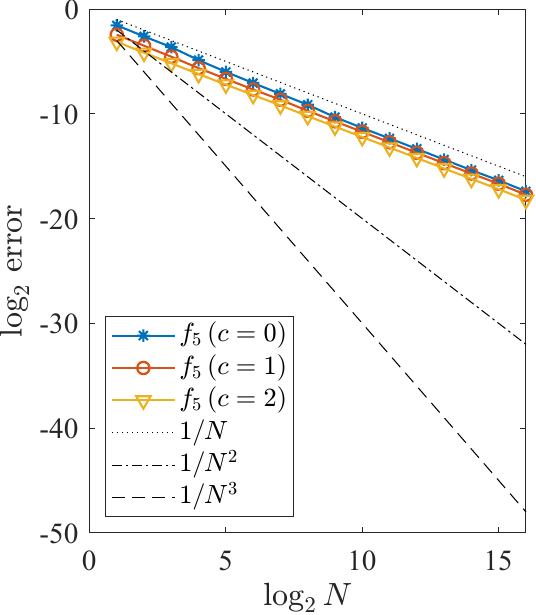}}
    \subfloat[Median scrambled Sobol']{\includegraphics[width=0.32\textwidth]{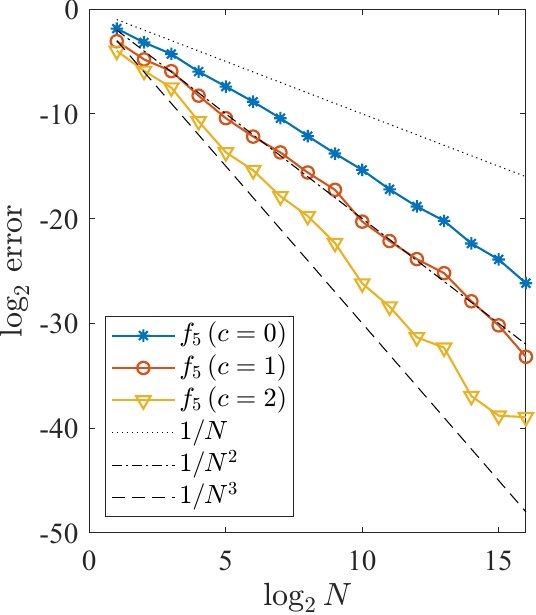}}
    \subfloat[Median polynomial lattices]{\includegraphics[width=0.32\textwidth]{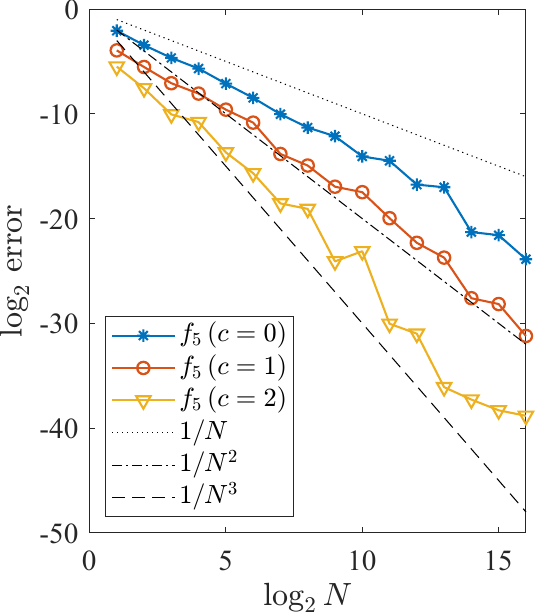}}
    \caption{Comparison of the 5-dimensional integration error by QMC rule using Sobol' points (left) and two median QMC rules using linearly scrambled Sobol' points (middle) and randomly chosen polynomial lattice point sets (right) for the infinitely smooth function $f_5$ with $c=0$ (blue), $c=1$ (orange) and $c=2$ (yellow), respectively.}
    \label{fig:test_5d_infinite}
\end{figure}
\end{example}

\section*{Acknowledgments}
The authors would like to thank the two referees for their constructive comments and suggestions. We are also grateful to Shu Tezuka for providing valuable feedback on the initial manuscript.

\bibliographystyle{siamplain}
\bibliography{ref.bib}
\end{document}